\documentclass{article}[12pt]
\usepackage[english]{babel}
\usepackage[utf8]{inputenc}
\usepackage{graphicx}
\usepackage{amsmath}
\usepackage{amssymb}
\usepackage{amsthm}
\usepackage{bbm}
\usepackage{amsfonts}
\usepackage{comment}
\usepackage{array}
\usepackage{amssymb}
\usepackage{mathrsfs}  
\usepackage{mathtools}
\usepackage{enumerate}
\usepackage{enumitem}
\usepackage{hyperref}
\usepackage{bbm}
\usepackage{dsfont}
\usepackage[top=2cm, bottom=2cm, left=2cm, right=2cm]{geometry}
\usepackage{graphicx}
\usepackage{tikz}
\usepackage{xcolor}
\usepackage{pgfplots}
\pgfplotsset{compat=1.7}
\usepackage{float}
\usepackage{eurosym}
\newtheorem{remark}{Remark}[section]

\numberwithin{equation}{section} 

\newtheorem{proposition}{Proposition}[section]
\newtheorem{lemma}{Lemma}[section]
\newtheorem{definition}{Definition}[section]
\newtheorem{theorem}{Theorem}[section]

\DeclareMathOperator{\argmin}{argmin}
\DeclareMathOperator{\argmax}{argmax}

\DeclareMathOperator{\supp}{supp}
\DeclareMathOperator{\Lip}{Lip}

\everymath{\displaystyle}

\def\F{\mathcal{F}}
\def\G{\mathcal{G}}
\def\W{\mathcal{W}}

\title{An optimal control problem for the continuity equation arising in smart charging}
\author{Adrien S{\'e}guret
\thanks{
PSL Research University, Université Paris-Dauphine, CEREMADE, Place de Lattre de Tassigny, F- 75016 Paris, France, \newline
OSIRIS department, EDF Lab,  Bd Gaspard Monge, 91120 Palaiseau, France, \newline
adrien.seguret@edf.fr}}

\date{}

\begin{document}

\maketitle
\begin{abstract}
This paper is focused on the mathematical modeling and solution of the optimal charging of a large population of identical plug-in electric vehicles (PEVs) with mixed state variables (continuous and discrete). A mean field assumption is formulated to describe the evolution interaction of the PEVs population. The optimal control of the resulting continuity equation of the mixed system under state constraints is investigated. We prove the existence of a minimizer. We then characterize the solution as the weak solution of a system of two coupled PDEs: a continuity equation and of a Hamilton-Jacobi equation.
We provide regularity results of the optimal feedback control.
\end{abstract}


\section{Introduction}\label{intro_section}
This article studies the optimal control of a first order continuity equation with a reaction term under state constraints.
Let us consider a finite time interval $[0,T]$ and a mixed state space equal to the product $[0,1]\times I$, where $I$ is a finite space, the cardinality of which is denoted by $\vert I\vert $. The continuity equation is given by: 
\begin{equation}
\label{fk}
\begin{array}{ll}
\partial_t m_i(t,s)+\partial_s(m_i(t,s)b_i(s))
=-\sum_{j\neq i}(\alpha_{i,j}(t,s)m_i(t,s)-\alpha_{j,i}(t,s)m_{j}(t,s))& (i,t,s)\in I\times (0,T)\times (0,1),
\\
m_i(0,s) =  m_i^0(s) & (i,s)\in I\times [0,1],
\end{array}
\end{equation}
where $m:[0,T]\to \mathcal{P}([0,1]\times I)$ is a curve of probability distribution, $m^0\in \mathcal{P}([0,1]\times I)$ is the given initial distribution, $b: [0,1]\to \mathbb{R}^{\vert I \vert}$ is a velocity field, and the control $\alpha:  [0,T]\times [0,1]\to \mathbb{R}_+^{\vert I \vert\times \vert I \vert }$ is a jump intensity. The notion of weak solution of  \eqref{fk} is specified in Definition \ref{weak_sol_cont_equ}. We further assume that $b$ vanishes at the boundary of $[0, 1]$ so that the mass conservation in $I\times [0,1]$ is guaranteed without forcing a boundary condition. The distribution $m$ is subject to the following congestion constraints:
\begin{equation}
\label{congestion_ineq_D}
m_i(t,[0,1])\leq D_i(t)\quad \forall(i,t)\in I\times [0,T],
\end{equation}
where $D:[0,T] \to (\mathbb{R}^\ast_+)^{\vert I \vert}$ is given. 
The objective function $J$ is defined as follows:
\begin{equation}
\label{obj_formulation}
J(m,\alpha) := \sum_{i\in I}\int_{0}^{T}\int_0^1\bigg( c_{i}(t,s)+\sum_{j\in I,j\neq i} L(\alpha_{ij}(t,s))\bigg)m_i(t,ds)dt + \sum_{i\in I}\int_0^1g_i(s)m_i(T,ds),
\end{equation}
where $c$ is a running cost, $L$ a function penalizing high values of the control $\alpha$ and $g$ a final cost.
Our purpose is to study the optimization problem:
\begin{equation}
\label{opt:J}
 \inf_{m,\alpha}\,J(m,\alpha), \;\mbox{where }(m,\alpha)\mbox{ is a weak solution of }\eqref{fk}\mbox{ and satisfies }\eqref{congestion_ineq_D}.
\end{equation}
Our work is initially motivated by
the optimal charging of a population of PEVs controlled by a central planner. 
The continuous variable $s\in [0,1]$ in \eqref{fk} represents the level of battery of a PEV. The discrete variable $i\in I$ represents the mode of charging (e.g. idling, charging, discharging, etc...). For any $(t,s,i)\in [0,T]\times [0,1]\times I$, the value $m_i(t,s)$ represents the proportion of PEVs at time $t$ at state $(s,i)$.
The given velocity field $b_i(s)$ denotes the power of
charge or discharge of a PEV in mode $i$ and with battery level $s$. 
For any $(t,s,i,j)\in [0,T]\times [0,1]\times I\times I$, the value $\alpha_{i,j}(t,s)$ denotes the jump intensity of PEVs from the state $(s,i)$ to the state $(s,j)$ at time $t$. The control $\alpha$ is required to be a non-negative measurable function.
The congestion constraint \eqref{congestion_ineq_D} avoids high demand of energy at each moment over the period.
By penalizing high values of $\alpha$, the cost function $L$ aims at avoiding the synchronization of jumps of the PEVs and the consequent instability of the electrical network.
The value $c_i(t,s)$ corresponds to the cost per PEV to pay at time $t\in [0,T)$ at state $(s,i)$; $g_i(s)$ is the final cost per PEV to pay at state $(s,i)$. 
Numerical results of Problem \eqref{opt:J} applied to smart charging can be found in \cite{seguret2021mean}.
Problem \eqref{opt:J} can be interpreted heuristically as an approximation of the limit case $N\to\infty$ of an optimal switching problem of $N$ PEVs. Combinatorial techniques as well as optimal control tools may fail to solve problems with large population of PEVs, due to the curse of dimensionality \cite{bellman2015adaptive}. 
To overcome these difficulties, a continuum of PEVs can be considered, leading to the techniques of the optimal control of PDE. The connection between the finite population problem and the mean field problem is addressed in a companion paper \cite{seguret2022mean} by the author.

\subsection{Contributions, methodology and literature}

\textbf{Contributions }
This paper makes three main contributions.

First we prove the existence of solutions of Problem \eqref{opt:J}.

Second, we derive optimality  conditions of Problem \eqref{opt:J}.
More precisely, let $H$ denote the Fenchel conjugate of $L$ and $H'$ denote the derivative of $H$, we show that, if $(m,\alpha)$ is a solution of \eqref{opt:J}, then there exists a pair $(\varphi, \lambda)$ such that, for any $i,j\in I$, $\alpha_{i,j}=H'(\varphi_i-\varphi_j)$, and  $(\varphi,\lambda,m)$ is a weak solution of the following system:
\begin{equation}
\label{syst_opt_cond_intro}
\left\{
\begin{array}{ll}
-\partial_t\varphi_i-b_i\partial_s\varphi_i-c_i-\lambda_i+\sum_{j\in I,j\neq i}H(\varphi_i-\varphi_j)= 0 
&\mbox{on }(0,T)\times (0,1)\times I
\\
\partial_t m_i+\partial_s(m_i b_i)
+\sum_{j\neq i}H'(\varphi_i-\varphi_j)m_i -H'(\varphi_j-\varphi_i)m_{j} =0 
&\mbox{on }(0,T)\times (0,1)\times I\vspace{0.2cm}\\
m_i(0,s) =  m_i^0(s),\,  \varphi_i(T,s) =  g_i(s)+\lambda_i(\{T\})
&\mbox{on } (0,1)\times I
\\
\int_0^1m_i(t,ds)-D_i(t)\leq 0,\,\lambda\geq 0
&\mbox{on }[0,T]\times I\\
\sum_{i\in I}\int_0^T\bigg(\int_0^1 m_i(t,ds)-D_i(t)\bigg)\lambda_i(dt)=0,
\end{array}
\right.
\end{equation}
The function $\varphi$ is the Lagrange multiplier associated with the dynamic constraint \eqref{fk}, and the measure $\lambda$ is associated with the congestion constraint \eqref{congestion_ineq_D}. 
The first equation in \eqref{syst_opt_cond_intro} is a backward Hamilton-Jacobi equation. The existence, uniqueness and characterization of weak solutions of the backward Hamilton-Jacobi equation are investigated in the paper. 
The second equation in \eqref{syst_opt_cond_intro} is a forward continuity equation, where the control $\alpha$, defined by $\alpha_{i,j}=H'(\varphi_i-\varphi_j)$, is optimal. The measure $\lambda$ is non-negative and finite. The last equality in \eqref{syst_opt_cond_intro} ensures that the congestion constraint \eqref{congestion_ineq_D} is satisfied.

Third, we obtain regularity property for any weak solution $(\varphi,\lambda,m)$ of \eqref{syst_opt_cond_intro}. 
We prove that, under suitable assumptions on the data $b,g$ and $c$, the multiplier $\varphi$ is in $\Lip([0,T]\times[0,1],\mathbb{R}^{\vert I \vert})+BV([0,T],\mathbb{R}^{\vert I \vert})$. As a result, the optimal control $\alpha$ is bounded and Lipschitz continuous in space uniformly w.r.t. the time variable and the measure $m$ is in $\Lip([0,1],\mathcal{P}([0,1]\times I))$. We show that if the initial distribution $m^0$ is absolutely continuous w.r.t. the Lebesgue measure and has a smooth density, then the measure $m$ is absolutely continuous w.r.t. the Lebesgue measure and has Lipschitz  continuous density. 
\newline

\textbf{Methodology }
The existence of an optimal solution is established by compactness arguments.
We adopt a duality approach to obtain \eqref{syst_opt_cond_intro}. More explicitly, we relax the dynamics \eqref{fk} and the congestion constraint \eqref{congestion_ineq_D}. The resulting relaxed problem is then expressed as the dual of another convex problem. We show that the system \eqref{syst_opt_cond_intro} is the optimality condition of Problem \eqref{opt:J}. 
\newline

\textbf{Literature }
Solving the optimal control of a Fokker-Planck equation by means of the duality theory has been well known since decades \cite{fleming1989convex, vinter1993convex}.
Our work follows the method developed in the seminal work by Benamou and Brenier 
\cite{benamou2000computational} for optimal transport problems. In \cite{benamou2000computational}, a continuity equation is controlled with initial and final constraints; optimality conditions are obtained as a system of PDEs close to \eqref{syst_opt_cond_intro}. Similar method and results also in optimal transport are derived in \cite{cardaliaguet2013geodesics}. More recently, this approach was applied to solve on optimal control problem of a Fokker-Planck equation under state constraints in the $1$-Wasserstein space \cite{daudin2022optimal,daudin2023optimal}, where Lipschitz regularity results of the optimal control are proved.

The optimality conditions \eqref{syst_opt_cond_intro} typically arises in the Mean Field Game (MFG for short) Theory. This class of problems, introduced by Lasry and Lions \cite{lasry2006stat, lasry2006jeux,  lasry2007mean} and Huang, Malhamé and Caines \cite{huang2007large,huang2006large}, describes the interaction among a large population of identical and rational agents in competition. Mean Field Control (MFC for short) and MFG theories have been extensively used over the last few years as a mathematical tool in electrical engineering. More specifically, the optimal control of PDEs applied to smart charging can be found in \cite{le2015optimal,sheppard2017optimal}, and to the management of a population of thermostatically controlled loads in \cite{ghaffari2015modeling,moura2013modeling}.  

Conversely, the duality approach is close to the so-called  variational approach used in MFG theory in \cite{cardaliaguet2015mean}, where the weak solution of the MFG system is characterized as the minimizer of some optimal control of Hamilton-Jacobi and Fokker-Planck equations.
This approach allows to use optimization techniques to prove the existence and uniqueness of the solution of MFG and MFC problems. We refer to \cite{achdou2016mean,benamou2017variational,briceno2018proximal,cardaliaguet2015second,orrieri2019variational} and the references therein. 
Besides, the variational approach allows to apply optimization algorithms to numerically solve MFG problems \cite{ benamou2015augmented,briceno2019implementation, briceno2018proximal}. 

Note that different optimality conditions for control problems in the space of probability measures can be derived by using a kind of Pontryagin Maximum Principle \cite{bonnet2021necessary,bonnet2019pontryagin}.

A particularity of this paper is to deal with a congestion constraint \eqref{congestion_ineq_D} on the measure. Two kinds of congestion effects are explored in the MFG and MFC frameworks. On the one hand, ``soft congestion" increases the cost of velocity of the agents in areas with high density. On the other hand, ``hard congestion" constraints impose density constraints, e.g. $m\leq \bar{m}$ at any point $(t,s)$. The variational approach yields good results  when applied to MFC \cite{achdou2016mean} and MFG with ``soft congestion" in a stationary framework \cite{evangelista2018first}, 
as well as to MFG problems dealing with 
``hard congestion" constraints. This was first investigated in \cite{santambrogio2012modest} where the density of the population did not exceed a given threshold, then in \cite{meszaros2015variational} where stationary second order MFGs were considered. In \cite{cardaliaguet2016first}, a price which is imposed on the saturated zone to make the density satisfy the constraints is obtained. In the same vein as the work of Benamou and Brenier \cite{benamou2000computational}, ``hard congestion" constraints were examined in optimal transport \cite{buttazzo2009optimization}.
We highlight that our paper deals with aggregate ``hard congestion" constraints on the measure $m$ \eqref{congestion_ineq_D}, i.e., our constraint is less restrictive than a constraint of the type $m\leq \bar{m}$ a.e.. 

Besides, we consider a mixed state space with continuous and discrete state variable. This setting has seldom been investigated in the MFG literature. Indeed, the articles cited in the paragraphs above looked only at continuous state variables. The resulting continuity equation \eqref{fk}
contains a term of reaction, indicating mass transfers between states in $I$. Such PDEs also arised in \cite{annunziato2014connection} to model the mean field limit of Piecewise Deterministic Markov Proccesses (PDMP for short). 
The velocity was controlled in \cite{annunziato2014connection}. Here, we control the intensity of the jump $\alpha$, but the velocity $b$ is given. A MFG problem with discrete time and state space was explored in \cite{gomes2010discrete} and applied to socio-economic sciences in \cite{gomes2014socio}. The uniqueness of the solution of a finite state MFG was discussed in \cite{bayraktar2021finite} and the convergence of the $N$-player game to the mean field model as $N\to \infty$  was obtained in \cite{gomes2013continuous}.
Mixed state spaces in a MFG framework were studied in \cite{firoozi2017mean,firoozi2022class}, where a major player can switch his state on a finite state space and minor players decide their stopping time.
A MFG problem in a finite state space and discrete time settings with ``hard congestion" was studied in \cite{bonnans2023discrete}, also by variational methods.

Some of our regularity results, namely that $\alpha$ is Lipschitz continuous w.r.t. the space variable $s$, and that the density of $m$ w.r.t. the Lebesgue measure is Lipschitz continuous when $m^0$ is absolutely continuous w.r.t. the Lebesgue measure, are unusual.
We believe that it is mainly due to the linearity of the Hamilton-Jacobi equation w.r.t. $\partial_s \varphi$.
These results will be used in a companion work \cite{seguret2022mean} to quantify the mean field limit of the model.
The time regularity of
$\varphi$ may not be improved as far as we have no more regularity results on $\lambda$. The function $\varphi$ is discontinuous at each atom of the measure $\lambda$.
Regularity results about the multiplier of the density constraint can be found in the literature: in \cite{cardaliaguet2016first}, the authors showed some BV estimates on the pressure, whereas $L^\infty$ estimates for the price were proved in \cite{daudin2023optimal} and \cite{lavenant2019new} in the special case of a quadratic Hamiltonian in a MFG problem.  
{\color{black} In \cite{bonnet2021intrinsic}, sufficient conditions are provided to obtain Lipschitz regularity results, w.r.t. the space variables, of the optimal control in a deterministic framework with continuous state variables. These results are closely related to those given in the articles \cite{gangbo2022global,gangbo2015existence,mayorga2020short}, which deal with first order MFGs.} 
The Sobolev regularity of the weak solutions of a first order MFG system was studied in \cite{santambrogio2018regularity} using duality arguments, in \cite{graber2018sobolev} for MFG problems with local coupling terms and in \cite{graber2019planning} for planning MFG. The existence of classical solutions to first order MFG problems was first investigated in \cite{lions2007theorie} and improved in \cite{munoz2022classical}.

\textbf{Organization of the paper}
The paper is organized as follows. In Section \ref{assum_results}, we present our assumptions and main results. The existence of a solution of Problem \eqref{opt:J} is established in Section \ref{prob_formulation_section}. In Section \ref{hjb_section}, which is independent of the other sections, we show the existence, uniqueness and regularity of weak solution of the Hamilton-Jacobi equation in \eqref{syst_opt_cond_intro}. 
In Section \ref{dual_prob_section}, we return to Problem \eqref{opt:J} and develop the duality approach. We formulate its Lagrangian relaxation and show it is the dual problem of another convex problem.
We obtain the optimality conditions \eqref{syst_opt_cond_intro} of Problem \eqref{opt:J} in Section \ref{charac_mini}. The Lipschitz continuity of the value of Problem \eqref{opt:J} w.r.t. the data $(m^0,D)$ is proved in Section \ref{lip_value_data}.
Basic statements about weak solutions of \eqref{fk} and technical results useful for Section \ref{hjb_section} are postponed in Appendix \ref{appendices_section}.

\section{Assumptions and main results}\label{assum_results}

\subsection{Notations and Assumptions}
\textbf{Notations} The space of Borel, positive and bounded measures on a space ${A}$ is denoted by $\mathcal{M}^+({A})$ and the space of Borel probability measures on a space ${A}$ is denoted by $\mathcal{P}({A})$. For any measure $\mu\in \mathcal{M}^+([0,T])$ and $0\leq t_1<t_2\leq T$, we set $\textstyle \int_{t_1}^{t_2}\mu(dt):=\mu([t_1,t_2])$.
The space of vector-valued, of dimension $d\in \mathbb{N}^+$ (resp. $d\times d$ with $d\in \mathbb{N}^+$), positive and bounded measures on a space ${A}$ is denoted by $\mathcal{M}^+({A},\mathbb{R}_+^d)$ (resp. $\mathcal{M}^+({A},\mathbb{R}_+^{d\times d})$).
  Given a set $\mathcal{S}$, for any function $f :\mathcal{S}\to \mathbb{R}^{\vert I \vert }$, $f_i(x)$ denotes the $i^{th}$ component of $f(x)$, for any $(x,i)\in \mathcal{S}\times I$. Similarly, for any  function $g :\mathcal{S}\to \mathbb{R}^{\vert I \vert \times \vert I \vert }$, $g_{i,j}(x)$ denotes the  element at the $i^{th}$ row and $j^{th}$ column of $g(x)$, for any $(s,i,j)\in \mathcal{S}\times I^2$.
  
  For any probability measure $\mu\in \mathcal{P}(\mathcal{S}\times I)$, we use the notation $\mu_{i}(S):=\mu(S\times\{i\})$ for any $(i,S)\in I\times \mathcal{B}(\mathcal{S})$, where $\mathcal{B}(\mathcal{S})$ denotes the Borel $\sigma$-algebra.
 By the so-called disintigration theorem \cite[Theorem 5.3.1]{ambrosio2008gradient}, for any probality measure $\mu\in  \mathcal{P}(\mathcal{S}\times I)$, $\mu_i$ is in $\mathcal{M}^+(\mathcal{S})$.
 
 For any open subset $\mathcal{O}$ of $\mathbb{R}^d$, the spaces $C^0(\mathcal{O}),C^1(\mathcal{O}) $ and $C^\infty_c(\mathcal{O})$ denote respectively the set of continuous functions, of continuously differentiable functions, and of test functions compactly supported defined on $\mathcal{O}$.
If $\mathcal{S}$ is a metric space, let  $\Lip(\mathcal{S})$  denote  the  vector space of
bounded and Lipschitz continuous maps $f:\mathcal{S}\to \mathbb{R}$. The space $\Lip([0,T]\times [0,1],\mathbb{R}^{\vert I \vert})+BV((0,T),\mathbb{R}^{\vert I \vert})$ denotes the set of functions with values in $\mathbb{R}^{\vert I \vert}$ that are the sum of a function in $\Lip([0,T]\times [0,1],\mathbb{R}^{\vert I \vert})$ and a function in $BV((0,T),\mathbb{R}^{\vert I \vert})$, where $BV((0,T),\mathbb{R}^{\vert I \vert})$ is the set of vector-valued functions of bounded variation function defined on $(0,T)$ \cite[Definition 3.1]{ambrosio2000functions}, i.e. the set of functions $f$ in $L^1((0,T),\mathbb{R}^{\vert I \vert})$ such that there exists a finite vector-valued Radon measure $\nu$ in $\mathcal{M}((0,T),\mathbb{R}^{\vert I\vert })$  satisfying for any $\phi\in C_c^\infty((0,T),\mathbb{R}^{\vert I \vert})$
\begin{equation*}
    \sum_{i\in I}\int_0^Tf_i(t)\partial_t\phi_i(t)dt=-\sum_{i\in I}\int_0^T\phi_i(t)\nu_i(dt).
\end{equation*}
We denote by $\W$ the $1$-Wasserstein distance on $\mathcal{P}([0,1]\times I)$, defined by 
\begin{equation*}
      \W(\mu,\rho):=\sup\,\bigg\{\sum_{i\in I}\int_{0}^1\varphi_i(\mu_i-\rho_i)\,\vert\,\varphi\mbox{ is }1-\mbox{Lipschitz from }[0,1]\times I\mbox{ to }\mathbb{R}\bigg\}.
\end{equation*}
We recall that if a function $\varphi$ is $1$-Lipschitz continuous from $[0,1]\times I$ to $\mathbb{R}$, then $\vert \varphi_i(x)-\varphi_j(x)\vert \leq 1$ for any $i,j\in I$. 
The space $C^0([0,T],\mathcal{P}([0,1]\times I))$ is the set of continuous functions from $[0,T]$ to $\mathcal{P}([0,1]\times I)$, where $\mathcal{P}([0,1]\times I)$ is endowed with the distance $\W$.
For any $\mu \in C^0([0,T],\mathcal{P}(\mathbb{R}))$, let
\begin{equation*}
    L^2_{\mu\otimes dt}([0,T]\times \mathbb{R})   :=\bigg\{f:[0,T]\times \mathbb{R}\mapsto \mathbb{R},\, \int_0^T\int_0^1f(t,s)^2\mu(t,ds)dt<+\infty\bigg\},
\end{equation*}
where $\mu\otimes dt$ denotes the product measure between $\mu$ and the Lebesgue measure on $\mathbb{R}_+$.
The dual of a normed space $X$ is denoted by $X^\ast$. {\color{black}For any $x\in X$ and $y\in X^\ast$, we define the pairing $\langle y,x\rangle_{X^\ast,X}$ by $\langle y,x\rangle_{X^\ast,X}:=y(x)$.} For any real valued function $F$ defined on $\mathbb{R}$, we denote by $\pmb{F}$ the Fenchel conjugate of $F$, defined on $\mathbb{R}$ by
\begin{equation*}
    \pmb{F}(y):=\underset{x\in \mathbb{R}}{\sup}\,xy-F(x).
\end{equation*}

 We consider the space 
 \begin{equation}\label{def_omega}
     \Omega:=\mathcal{P}( [0,1]\times I)\times C^0( [0,T],(\mathbb{R}_+^\ast)^{\vert I \vert}).
 \end{equation}
For any $\varepsilon>0$, we define the subspace $\Omega_\varepsilon$ of $\Omega$, by
\begin{equation}
\label{def_omega_eps}
    \Omega_\varepsilon:=\bigg\{(\mu,f)\in \Omega\,\vert \, \varepsilon<\inf_{t\in [0,T],i\in I}f_i(t)-\mu_i([0,1])\bigg\}.
\end{equation}
 We consider the following spaces for any $\delta>0$:
    \begin{equation}\label{def_space_delta}
    \begin{array}{ll}
     \mathcal{M}_{\delta}^+([0,T],\mathbb{R}_+^{\vert I \vert })   & :=  \bigg\{\mu\in \mathcal{M}^+([0,T],\mathbb{R}_+^{\vert I \vert}) \,\vert\,\sum_{i\in I}\mu_i([0,T])\leq \delta\bigg\},\vspace{0.2cm}\\
      C^0_\delta([0,T],\mathbb{R}_+^{\vert I \vert})&:= \bigg\{f\in C^0([0,T],\mathbb{R}_+^{\vert I \vert}) \,\vert\,\int_0^T\sum_{i\in I}f_i(t)dt\leq \delta\bigg\}.
    \end{array}
    \end{equation}
For any $\delta>0$, the space $C^0_\delta([0,T],\mathbb{R}_+^{\vert I \vert})$ can be considered as a subspace of $\mathcal{M}_{\delta}^+([0,T],\mathbb{R}_+^{\vert I \vert })$ in the sense that, for any $f\in C^0_\delta([0,T],\mathbb{R}_+^{\vert I \vert})$, we have $f\mathcal{L}\in  \mathcal{M}_{\delta}^+([0,T],\mathbb{R}_+^{\vert I \vert })$, where $\mathcal{L}$ is the Lebesgue measure on $[0,T]$.
\newline

\textbf{Assumptions} The following assumptions are in force throughout the paper.
\begin{enumerate}
\item \label{hyp_on_b} For any $i\in I$, $b_i\in C^1(\mathbb{R})$ with $b_i(s)=0$ for any $s \not\in (0,1)$. 
\item  \label{hyp_on_m_0} The initial distribution of \eqref{fk} $m^0\in \mathcal{P}(\mathbb{R}\times I)$ is such that $\supp(m_i^0)\subset [0,1]$ for any $i\in I$ . 
\item \label{hyp_on_D} There exists $\varepsilon^0>0$ such that the parameter $D$ of \eqref{congestion_ineq_D} and $m^0$ satisfy $(m^0,D)\in \Omega_{\varepsilon^0}$.
\item \label{hyp_on_c_g}For any $i\in I$, it is assumed that $c_i\in C^1([0,T]\times [0,1])$ and $g_i\in C^1([0,1])$.
\item \label{assum_l} $L:\mathbb{R}\to \bar{\mathbb{R}}$ is a convex function, defined by:
    $$
    L(x):=\left\{
    \begin{array}{ll}
       l(x)  & \mbox{if }x>0,  \\
       0  & \mbox{if }x=0,\\
       +\infty & \mbox{otherwise,}
    \end{array}\right.
    $$ where $l\in C^1(\mathbb{R_+},\mathbb{R_+})$ is an increasing strongly convex function satisfying $\underset{x\to 0}{\lim}l(x) = 0$, and bounded from above by a quadratic function. More explicitly, there exists $C>0$ such that for any $x\in \mathbb{R_+}$:
    \begin{equation*}
     {\color{black}\frac{x^2}{C} }\leq  l(x)\leq C( x^2 +1) ,
    \end{equation*}
    where the first inequality is due to the strong convexity of $l$.
\end{enumerate}
The function $H$ being the Fenchel conjugate of $L$, by Assumption \ref{assum_l}, $H$ is  non-decreasing and non-negative. Further, since $L$ is strongly convex on $\mathbb{R}_+$ and $H'$ equals $0$ on $\mathbb{R}^-$, $H'$ is Lipschitz continuous on $\mathbb{R}$ \cite[Theorem  E.4.2.1]{hiriart2004fundamentals}.

By Assumption \ref{hyp_on_b}, for any $s\in[0,1]$, $i\in I$ and $t\in [0,T]$, there exists a unique $S_i^{t,s}$ satisfying the ODE below
\begin{equation}
\label{ODE}
\begin{cases}
     \frac{\mbox{d}S_i^{t,s}(\tau)}{\mbox{d}\tau}=b_i(S_i^{t,s}(\tau)),\quad \tau\in [0,T], \vspace{0.2cm}\\
    S_i^{t,s}(t)=s  . 
\end{cases}
\end{equation}
\begin{remark}
The main role of Assumptions \ref{hyp_on_b} and \ref{hyp_on_m_0} is to ensure that the support of the weak solution of \eqref{fk} is contained in $[0,1]$ over the period $[0,T]$ (cf. Lemma \ref{supportm} in Appendix \ref{appendices_section}).
Assumption \ref{hyp_on_D} provides an estimate on $\sum_{i\in I}\lambda_i([0,T])$ for any weak solution $(\varphi, \lambda, m)$ of \eqref{syst_opt_cond_intro} and ensures that the control $\alpha=0$ is an admissible control. Correspondingly, the feasible set of Problem \eqref{opt:J} is not empty. Regularity results of the weak solutions of the system \eqref{syst_opt_cond_intro} are derived thanks to the assumptions formulated on $c$ and $g$ in Assumption \ref{hyp_on_c_g}. 
\end{remark}

\subsection{Main results}
We introduce, for a given $\lambda\in \mathcal{M}^+([0,T],\mathbb{R}_+^{\vert I \vert})$, the Hamilton-Jacobi equation on $(0,T)\times (0,1)\times I$:
\begin{equation}
\label{hjb_0}
\begin{array}{ll}
- \partial_ t\varphi_i(t,s)-b_i(s)\partial_s\varphi_i(t,s)-c_i(t,s)-\lambda_i(t)+\sum_{j\in I,j\neq i}H((\varphi_j-\varphi_i)(t,s))= 0 & (t,s,i)\in (0,T)\times (0,1)\times I,\\
\varphi_i(T,s)= g_i(s) +\lambda_i(\{T\})& (s,i)\in [0,1]\times I,
\end{array}
\end{equation}
where $\lambda_i(\{T\}):=\lambda_i([T,T])$ in the final condition of \eqref{hjb_0} denotes the measure of $\{T\}$ w.r.t. $\lambda_i$ and the term $\lambda_i(t)$ in the Hamilton-Jacobi equation of \eqref{hjb_0} is an abuse of notation. 
We are mainly interested in weak solutions of equation \eqref{hjb_0} (see Definition \ref{def_weak_sol_hjb}).
We define the function $\tilde{A}$ for any $(\varphi, \lambda)\in \big(\Lip([0,T]\times [0,1],\mathbb{R}^{\vert I \vert})+BV([0,T],\mathbb{R}^{\vert I \vert})\big)\times \mathcal{M}^+([0,T],\mathbb{R}_+^{\vert I \vert})$ by:
\begin{equation}
\label{def_tilde_A_0}
\tilde{A}(\varphi,\lambda):=
\sum_{i \in I}\int_0^1  - \varphi_i(0^+,s)m_i^0(ds) + \int_0^TD_i(t)\lambda_i(dt).
\end{equation}
The following theorem summarizes the main results of the paper.
\begin{theorem}\label{main_results}
Problem \eqref{opt:J} has a solution.
Furthermore, the minimizers have the following properties:
\begin{enumerate}
\item\label{min_imp_wea} If $(m,\alpha)$ is a minimizer of Problem \eqref{opt:J} and $(\varphi, \lambda)\in \Big(\Lip([0,T]\times [0,1],\mathbb{R}^{\vert I \vert})+BV([0,T],\mathbb{R}^{\vert I \vert})\Big)\times \mathcal{M}^+([0,T],\mathbb{R}_+^{\vert I \vert})$ is such that $\tilde{A}(\varphi,\lambda)=-J(m,\alpha)$, then $(\varphi, \lambda, m)$ is a weak solution of \eqref{syst_opt_cond_intro} in the sense of Definition \ref{weak_sol_syst}, and $\alpha_{i,j}=H'(\varphi_i-\varphi_j)$ on $\{m_i>0\}$ for any $i,j\in I$.
\item \label{wea_imp_min} Conversely, if $(\varphi,\lambda,m)$ is a weak solution of \eqref{syst_opt_cond_intro} in the sense of Definition \ref{weak_sol_syst}, then there exists $\alpha$, defined for any $i,j\in I$ by: $\alpha_{i,j}:=H'(\varphi_i-\varphi_j)$ on $\{m_i>0\}$, such that $(m,\alpha)$ is a minimizer of  \eqref{opt:J} and $\tilde{A}(\varphi,\lambda)=-J(m,\alpha)$.
\item \label{min_imp_reg}If $(m,\alpha)$ is a minimizer of Problem \eqref{opt:J}, then for any $i,j\in I,$ $\alpha_{i,j}$ and $\partial_s\alpha_{i,j}$ are both in $L^\infty([0,T]\times [0,1])$, and $m$ is in $\Lip([0,T],\mathcal{P}( [0,1]\times I))$.
\end{enumerate}
\end{theorem}
The existence of a solution of Problem \eqref{opt:J} is stated in Lemma \ref{Existence_sol_prob_init}.
The characterization of a solution, i.e., Theorem \ref{main_results}.\ref{min_imp_wea}-\ref{wea_imp_min}, is given by Theorem \ref{optimality_conditions}. For this, we adopt a duality approach which is developed in Section \ref{dual_prob_section}. In particular, we introduce a convex problem, the dual of which is Problem \eqref{opt:J} up to a change of variable (cf. Theorem \ref{pb_in_duality}). The Lipschitz continuity of $m$ stated in Theorem \ref{main_results}.\ref{min_imp_reg} is deduced from the regularity of $\varphi$, which is derived in Section \ref{hjb_section}, and Proposition \ref{weak_sol_fk_coro}.
The proof of Theorem \ref{main_results} is finally given in Section \ref{proof_main_th}.

Further regularity results on the solution of Problem \eqref{opt:J} are obtained with additional conditions on the initial distribution $m^0$.
\begin{proposition}\label{reg_m_sol}
If the initial distribution $m^0$ is absolutely continuous w.r.t. the Lebesgue measure with a density in $C^1([0,1],\mathbb{R}^{\vert I \vert})$, then any solution $(m,\alpha)$ of Problem \eqref{opt:J} is such that $m$ is absolutely continuous w.r.t. the Lebesgue measure on $[0,T]\times [0,1]$, and has a density in $\Lip([0,T]\times [0,1],\mathbb{R}^{\vert I \vert})$. 
\end{proposition} 
The proof of Proposition \ref{reg_m_sol} is given in Section \ref{proof_main_th}.

\section{Existence of an optimal solution}\label{prob_formulation_section}

We introduce in this section a convex problem that is equivalent to \eqref{opt:J}. Standard compactness arguments are used to show the existence of an optimal solution (see for e.g. \cite[Theorem 3.1]{daudin2022optimal}).
The definition of a weak solution of \eqref{fk} is specified below.

\label{prob_formulation}
\begin{definition}
\label{weak_sol_cont_equ}
A pair $(m,\alpha)$ satisfies \eqref{fk} in the weak sense if $m$ is in $C^0([0,T],\mathcal{P}(\mathbb{R}\times I))$, for any $i,j\in I$ with $i\neq j$, it holds that $\alpha_{i,j}\in L_{m_i\otimes dt}^2([0,T]\times \mathbb{R})$ and for any test function $\phi \in C^\infty_c([0,T]\times \mathbb{R},\mathbb{R}^{\vert I \vert})$, we have:
\begin{equation*}
\begin{array}{l}
\sum_{i\in I}\int_{\mathbb{R}} \phi_i(T,s)m_i(T,ds) - \phi_i(0,s)m_i^0(ds)   \\
  =  \int_0^T\int_{\mathbb{R}} \sum_{i\in I}
  (\partial_t\phi_i(t,s) + b_i(s)\partial_s\phi_i(t,s))m_i(t,ds) + 
  \sum_{j\in I, j\neq i}(\phi_j(t,s)-\phi_i(t,s))\alpha_{i,j}(t,s)m_i(t,ds)dt.
\end{array}
\end{equation*}
\end{definition}
\begin{remark}
Using Assumptions \ref{hyp_on_b} and \ref{hyp_on_m_0}, Lemma \ref{supportm} in Appendix \ref{appendices_section} states that, for any weak solution $(m,\alpha)$ of \eqref{fk} in the sense of Definition \ref{weak_sol_cont_equ}, the measure $m_i(t,\cdot)$ has its support contained in $[0,1]$ for any $(t,i)\in [0,T]\times I$. Thus, we will consider throughout the paper  only weak solutions $(m,\alpha)$ of \eqref{fk} satisfying $m(t,\cdot)\in \mathcal{P}([0,1]\times I)$ for any $t\in [0,T]$.
\end{remark}
Problem \eqref{opt:J} being not convex  w.r.t. the variables $(m,\alpha)$, we make a change of variables $E_{i,j}:=\alpha_{i,j} m_i$, for any $i,j\in I$ with $i\neq j$. We now rewrite the continuity equation \eqref{fk}:
\begin{equation}
\label{evoX}
\begin{array}{ll}
\partial_t m_i(t,s)+\partial_s(m_i(t,s)b_i(s))
=-\sum_{j\in I,j\neq i}(E_{i,j}(t,s)- E_{j,i}(t,s))& (i,t,s)\in I\times (0,T)\times (0,1)\\
m_i(0,s) =  m_i^0(s)& (i,s)\in I\times [0,1],
\end{array}
\end{equation}
where $E_{i,j}\in\mathcal{M}^+([0,T]\times[0,1])$ is such that $E_{i,j}\ll m_i\otimes dt$, with $\frac{\mbox{d}E_{i,j}}{ \mbox{d} m_i\otimes dt} = \alpha_{i,j}$ and $\frac{\mbox{d}E_{i,j}}{ \mbox{d} m_i\otimes dt} \in L_{m_i\otimes dt}^2([0,T]\times [0,1])$.
For any $(m^0,D)\in \Omega$ ($\Omega$ is defined in \eqref{def_omega}) satisfying Assumption~\ref{hyp_on_D}, we introduce the set:
\begin{equation}
\begin{array}{ll}
\label{def_set_CE}
\mathcal{S}(m^0,D):=&
\bigg\{ (m,E)\mbox{ such that: }E_{i,j}\ll m_i\otimes dt\, \,\forall i,j\in I,\,\alpha\geq 0\mbox{ and }(m,\alpha) \mbox{ satisfies }\eqref{evoX}.\\
&\,\mbox{ in the weak sense, where }\alpha_{i,j}:= \frac{\mbox{d}E_{i,j}}{ \mbox{d} m_i\otimes dt},\,\mbox{ and }m
\mbox{ satisfies }\eqref{congestion_ineq_D} \bigg\}.
\end{array}
\end{equation}
By the disintegration theorem \cite[Theorem 5.3.1]{ambrosio2008gradient}, for any $(m,E)\in \mathcal{S}(m^0,D)$ and $i,j\in I$, we can write $E_{i,j}(dt,ds)=E_{i,j}(t,ds)dt$.
From Assumption \ref{hyp_on_D}, the set $\mathcal{S}(m^0,D)$ is not empty. Indeed, denoting by $(m,0)$ a weak solution of \eqref{fk} with control $\alpha\equiv 0$, the distribution $m$ satisfies that, for any $(t,i)\in [0,T]\times I$,  $m_i(t,[0,1])=m_i^0([0,1])<D_i(t)$. Thus, $(m,0)\in \mathcal{S}(m^0,D)$.
We define the function $\tilde{J}$ on $\mathcal{S}(m^0,D)$ by:
\begin{equation}
\label{def_B}
\tilde{J}(m,E): =
\sum_{i\in I} \int_0^T\int_0^1 c_i(t,s)  m_i(t,ds) dt+ \int_0^T\int_0^1 \sum_{i,j\in I,j\neq i}L \Big(\frac{\mbox{d}E_{i,j}}{ \mbox{d} m_i\otimes dt}(t,s)\Big)m_i(t, ds)dt + \sum_{i\in I}\int_0^1 g_i(s)m_i(T,ds),
\end{equation}
where the function $L$ is defined in Assumption \ref{assum_l}. Since $m(t)$ is a probability measure for any $t\in[0,T]$, by Assumption \ref{hyp_on_c_g}, the first and last integrals in \eqref{def_B} are well defined. Since $\frac{\mbox{d}E_{i,j}}{ \mbox{d} m_i\otimes dt} \in L_{m_i\otimes dt}^2([0,T]\times [0,1])$ and $l$ is bounded by a quadratic function according to Assumption \ref{assum_l}, the second integral in \eqref{def_B} is also well defined. Thus, for any $(m,E)\in \mathcal{S}(m^0,D)$, the quantity $\tilde{J}(m,E)$ is finite.
The following optimization problem is considered:
\begin{equation}
\label{problemE}
 \underset{(m,E)\in \mathcal{S}(m^0,D)}{\inf}\,\tilde{J}(E,m).
\end{equation}
For any $\gamma >0$, we define the subset $\mathcal{S}_\gamma(m^0,D)$ of $\mathcal{S}(m^0,D)$ by:
\begin{equation}
\label{ineqGamma0}
\mathcal{S}_\gamma(m^0,D):=\bigg\{(m,E)\in \mathcal{S}(m^0,D)\,\Big\vert\,  \sum_{(i,j)\in I,i\neq j}\int_0^T\int_0^1 L \Big(\frac{\mbox{d}E_{i,j}}{ \mbox{d} m_i\otimes dt}(t,s)\Big)m_i(t, ds)dt \leq \gamma\bigg\}.
\end{equation}
For any  $(m,E)\in\mathcal{S}_\gamma(m^0,D)$, the next lemma states that $m$ is $\frac{1}{2}$-Hölder continuous, i.e. there exists $C>0$ such that for any $t_1,t_2\in [0,T]$
\begin{equation*}
   \mathcal{W}(m(t_1),m(t_2))\leq C\sqrt{\vert t_1-t_2\vert}.
\end{equation*} 
\begin{lemma}
\label{UnifCont}
For any $\gamma> 0$, there exists a positive constant $C_\gamma$ such that, for any $(m,E)\in \mathcal{S}_\gamma(m^0,D)$, $m$ is $\frac{1}{2}$-Hölder continuous with constant $C_\gamma$ from $[0,T]$ to $\mathcal{P}([0,1]\times I)$ (endowed with the $1$-Wasserstein distance $\mathcal{W}$).
\end{lemma}

\begin{proof}
Let $\varphi:[0,1]\times I\mapsto \mathbb{R}$ be globally $1-$Lipschitz continuous and continuously differentiable w.r.t. the continuous variable. Since $(m,E)$ is a weak solution of \eqref{evoX}, we have, for any $t_1,t_2\in [0,T]$ with $t_1<t_2$,
\begin{equation*}
    \sum_{i\in I}\int_0^1\varphi_i(s)\big(m_i(t_2,ds)-m_i(t_1,ds)\big)=
   \sum_{i\in I} \int_{t_1}^{t_2} \int_0^1b_i(s)\partial_s\varphi_i(s) m_i(t,ds)dt
+
 \sum_{i,j\in I,j\neq i}\int_{t_1}^{t_2} \int_0^1(\varphi_j(s)-\varphi_i(s))E_{i,j}(t,ds)dt.
\end{equation*}
Since $\varphi$ is globally $1-$Lipschitz continuous, $\Vert \partial_s \varphi\Vert_\infty\leq 1$, $\vert \varphi_j(s)-\varphi_i(s)\vert \leq 1$ and $E$ is a non-negative measure, then previous equality becomes
\begin{equation}\label{ineq_for_holder}
  \left\vert  \sum_{i\in I}\int_0^1\varphi_i(s)\big(m_i(t_2,ds)-m_i(t_1,ds)\big)\right\vert\leq \vert t_2-t_1\vert \Vert b\Vert_\infty
+
 \sum_{i,j\in I,j\neq i}\int_{t_1}^{t_2} \int_0^1E_{i,j}(t,ds)dt.
\end{equation}
We need to find an upper bound of $ \sum_{i,j\in I,j\neq i}\int_{t_1}^{t_2} \int_0^1E_{i,j}(t,ds)dt$. By Cauchy--Schwartz inequality, one has
\begin{equation*}
\left(\int_{t_1}^{t_2}\int_0^1 E_{i,j}(t,ds)dt \right)^2
     \leq(t_2-t_1)\int_{t_1}^{t_2} \int_0^1\sum_{i,j\in I,j\neq i}\left(\frac{\mbox{d}E_{i,j}}{ \mbox{d} m_i\otimes dt}(t,s)\right)^2m_i(\tau,ds)d\tau,
\end{equation*}
{\color{black}
and by Assumption \ref{assum_l} on $L$, 
\begin{equation}\label{bound_E}
\begin{array}{ll}
\left(\int_{t_1}^{t_2}\int_0^1 E_{i,j}(t,ds)dt \right)^2
     & \leq C ({t_2}-{t_1})\int_0^T\int_{0}^1 {\color{black}\sum_{i,j\in I,j\neq i}}L \Big(\frac{\mbox{d}E_{i,j}}{ \mbox{d} m_i\otimes dt}(t,s)\Big)m_i(t, ds)dt
     \\& \leq C({t_2}-{t_1})\gamma,
\end{array}
\end{equation}
where the second inequality used that $(m,E)$ is in $S_\gamma(m^0,D)$ (defined in \eqref{ineqGamma0})}.
Then, by previous inequality, inequality \eqref{ineq_for_holder} becomes
\begin{equation*}
      \left\vert  \sum_{i\in I}\int_0^1\varphi_i(s)\big(m_i(t_2,ds)-m_i(t_1,ds)\big)\right\vert\leq \sqrt{\vert t_2-t_1\vert}\big(\sqrt{T} \Vert b\Vert_\infty +\sqrt{C\gamma }\big).
\end{equation*}
Setting $C_\gamma: =\big(\sqrt{T} \Vert b\Vert_\infty +\sqrt{C\gamma }\big)$, one deduces from previous inequality
\begin{equation*}
\mathcal{W}(m(t_2,\cdot),m({t}_1,\cdot))\leq 
C_\gamma\sqrt{\vert t_2-t_1\vert}.
\end{equation*}
 \end{proof}
By taking $t_1=0$ and $t_2=T$ in \eqref{bound_E}, one has for any $(m,E)$ in $\mathcal{S}_\gamma(m^0,D)$ that 
\begin{equation}
    \label{bound_mass_E}
   \sum_{i,j}E_{i,j}([0,T],[0,1])\leq \sqrt{CT\gamma}. 
\end{equation}
The next lemma is useful to show that any minimizing sequence of Problem \eqref{problemE} is relatively compact.
\begin{lemma}
\label{RelativCompact}
For any $\gamma>0$, the subset $\mathcal{S}_\gamma(m^0,D)$ is compact in $C([0,T],\mathcal{P}([0,1]\times I))\times \mathcal{M}^+([0,T]\times [0,1],\mathbb{R}^{\vert I \vert \times \vert I \vert })$, where $\mathcal{M}^+([0,T]\times [0,1],\mathbb{R}^{\vert I \vert \times \vert I \vert })$ is endowed with the weak$^\ast$ convergence.
\end{lemma}
\begin{proof}
Let $\{(m^n,E^n)\}_n$ be a sequence in $\mathcal{S}_\gamma(m^0,D)$. By inequality \eqref{bound_mass_E}, the sequence $\{E^n\}_n$ is relatively compact \cite[Theorem 1.59]{ambrosio2000functions} and converges weakly$^\ast$, up to a subsequence, to $\tilde{E}$ in $\mathcal{M}^+([0,T]\times [0,1],\mathbb{R}^{\vert I \vert \times \vert I \vert })$. By Lemma \ref{UnifCont}, the sequence $\{m^n\}_n$ is uniformly continuous and by Arzela–Ascoli theorem converges uniformly, up to a subsequence, to $\tilde{m}$ in 
$C\big([0,T],\mathcal{P}([0,1]\times I\big)$. By the linearity of the continuity equation \eqref{evoX} and the congestion constraint \eqref{congestion_ineq_D}, $(\tilde{m},\tilde{E})$ is a weak solution of \eqref{evoX} and satisfies \eqref{congestion_ineq_D}. By assumptions \ref{hyp_on_c_g} and \ref{assum_l}, the functional $\tilde{J}$, defined in \eqref{def_B}, is lower-semicontinuous (l.s.c. for short) \cite[Theorem 2.34]{ambrosio2000functions}. 
Thus, by the definition of $L$ in Assumption \ref{assum_l}, we have, for any $i,j\in I$, $\tilde{E}_{i,j}\ll \tilde{m}_{i}$ and $\gamma\geq \underset{n\to \infty}{\lim {\color{black}\inf}}\tilde{J}({m}^n,{E}^n)\geq \tilde{J}(\tilde{m},\tilde{E})$. Thus, $(\tilde{m},\tilde{E})$ is in $\mathcal{S}_\gamma(m^0,D)$.
\end{proof}

\begin{lemma}\label{Existence_sol_prob_init}
Problem \eqref{problemE} admits a solution. Consequently, Problem \eqref{opt:J} has a solution.
\end{lemma}
\begin{proof}
Existence of a solution of \eqref{problemE} is a consequence of Lemma \ref{RelativCompact} and of the l.s.c.  of  $\tilde{J}$  w.r.t. the topology induced by $C([0,T],\mathcal{P}([0,1]\times I))\times \mathcal{M}^+([0,T]\times [0,1],\mathbb{R}^{\vert I \vert \times \vert I \vert })$ (where $\mathcal{M}^+([0,T]\times [0,1],\mathbb{R}^{\vert I \vert \times \vert I \vert })$ is endowed with the weak$^\ast$ convergence). 
Since Problem \eqref{opt:J} is equivalent to Problem \eqref{problemE} up to a change of variable, the existence of a solution is straightforward.
\end{proof}

\section{Analysis of the Hamilton-Jacobi equation}\label{hjb_section}

The purpose of this section is to study weak solutions of the Hamilton-Jacobi equation \eqref{hjb_0} when $\lambda$ is in $\mathcal{M}^+([0,T],\mathbb{R}_+^{\vert I \vert})$. It is divided into three subsections.
Section \ref{cont_val_data} is devoted to the analysis of the following equation
\begin{equation}
\label{hjb}
\begin{array}{ll}
- \partial_ t\psi_i(t,s)-b_i(s)\partial_s\psi_i(t,s)-c_i(t,s)+\sum_{j\in I,j\neq i}H\bigg((\psi_i-\psi_j)(t, s)+\int_t^T(\lambda_i-\lambda_j)(dr)\bigg)= 0 & \mbox{on }(0,T)\times (0,1)\times I,\\
\psi_i(T,\cdot)= g_i & \mbox{on } [0,1]\times I.
\end{array}
\end{equation}
for a given $\lambda\in C^0_\delta([0,T],\mathbb{R}_+^{\vert I \vert})$. A classical solution $\psi$ of \eqref{hjb} associated with a function $\lambda$ in $C^0_\delta([0,T],\mathbb{R}_+^{\vert I \vert})$ is a function in $C^1([0,T]\times [0,1],\mathbb{R}^{\vert I \vert})$ satisfying \eqref{hjb} for any $(t,s,i)$ in $[0,T]\times [0,1]\times I$.
Then, equation \eqref{hjb} is studied for a given $\lambda\in \mathcal{M}^+_\delta([0,T],\mathbb{R}^{\vert I \vert})$ in Section \ref{lambda_measure_valued_data}. Finally, we obtain in Section \ref{existence_measure_weak_solution} the existence and the uniqueness of a weak solution to \eqref{hjb_0} by using the solution of \eqref{hjb}.

For any $\lambda\in \mathcal{M}^+([0,T],\mathbb{R}_+^{\vert I \vert})$, 
$\psi\in L^2([0,T]\times [0,1],\mathbb{R}^{\vert I \vert}),(t,\tau,s,i,j)\in [0,T]^2\times [0,1]\times I^2$, we consider in this section the function $H^\lambda$, defined by:
        \begin{equation}
        \label{def_h_lambda_measure}
    H^\lambda(i,j,t,\tau,s,\psi):=H\bigg((\psi_i-\psi_j)(\tau, S_i^{t,s}(\tau))+\int_\tau^T(\lambda_i-\lambda_j)(dr)\bigg),
    \end{equation}
where $S^{t,s}_i$ is the unique solution of \eqref{ODE}.

\subsection{The Hamilton-Jacobi equation for continuous data}\label{cont_val_data}

The main result of this subsection is the following.
\begin{proposition} \label{classical_sol}
For any $\lambda\in C^0_\delta([0,T],\mathbb{R}_+^{\vert I \vert})$, there exists a unique $\psi^\lambda\in C^1([0,T]\times [0,1],\mathbb{R}^{\vert I \vert})$ satisfying, for any $(t,s,i)\in [0,T]\times [0,1]\times I$,
\begin{equation}\label{def_integral_equation_2}
    \psi^\lambda_i(t,s)=\int_t^T\sum_{j\in I,j\neq i}-H^\lambda(i,j,t,\tau,s,\psi^\lambda) + c_i(\tau,S_i^{t,s}(\tau)) d\tau  + g_i(S_i^{t,s}(T)).
\end{equation}
In addition, $\psi^\lambda$ is the unique classical solution of \eqref{hjb} on $[0,T]\times [0,1]\times I$.
\end{proposition}
The next lemma states that a function $\psi$ satisfies \eqref{def_integral_equation_2} if and only if it is a classical solution of \eqref{hjb}.
\begin{lemma}\label{eq_int_pde}
For any $\lambda\in C^0([0,T],\mathbb{R}_+^{\vert I \vert})$ and $t_0\in [0,T)$, a function $\psi \in C^1((t_0,T]\times [0,1],\mathbb{R}^{\vert I \vert})$ satisfies equation \eqref{def_integral_equation_2} on $[t_0,T]\times [0,1]\times I$ if and only if it is a solution of \eqref{hjb} on $(t_0,T]\times [0,1]\times I$.
\end{lemma}
\begin{proof}
If $\psi \in C^1((t_0,T]\times [0,1],\mathbb{R}^{\vert I \vert})$ satisfies \eqref{def_integral_equation_2}, then $\psi(T,\cdot)=g$ and by computing the partial derivatives in space and time of $\psi$, one obtains that $\psi$ is a solution of \eqref{hjb} on $(t_0,T]\times (0,1)\times I$.

Conversely, if $\psi \in C^1((t_0,T]\times [0,1],\mathbb{R}^{\vert I \vert})$ is a solution of \eqref{hjb} then, by applying the method of characteristics \cite[Remark 8.1.5]{ambrosio2008gradient}, the function defined, for any $(t,s,i)\in (t_0,T]\times [0,1]\times I)$, by $\tau\mapsto \psi_i(\tau,S^{t,s}_i(\tau))$, satisfies on $(t_0,T)$:
\begin{equation}\label{dev_charac_method}
\begin{array}{ll}
  \frac{d}{d\tau}\psi_i(\tau,S^{t,s}_i(\tau))    &  =\partial_t\psi_i(\tau,S^{t,s}_i(\tau)) + b_i(S^{t,s}_i(\tau))\partial_s\psi_i(\tau,S^{t,s}_i(\tau))  \\
     & =-c_i(\tau,S^{t,s}_i(\tau))+\sum_{j\in I,j\neq i}H\bigg((\psi_i-\psi_j)(\tau,S^{t,s}_i(\tau))+\int_\tau^T(\lambda_i-\lambda_j)(r)(dr)\bigg),
\end{array}
\end{equation}
and $\psi_i(T,S^{t,s}_i(T))= g_i(S^{t,s}_i(T))$. By integrating, for any $t$ in $[t_0,T)$, equality \eqref{dev_charac_method} over $[t,T]$, one obtains that  $\psi$ is a solution of \eqref{def_integral_equation_2} on $[t_0,T]\times [0,1]\times I$.
\end{proof}

Before proving Proposition \eqref{classical_sol}, we need the following estimate on solutions of the Hamilton-Jacobi equation \eqref{hjb}.
\begin{lemma}[A priori estimate]\label{a_priori_estimate}
For any $\delta>0$, there exists $M>0$ such that,
for any $\lambda\in C_\delta^0([0,T],\mathbb{R}_+^{\vert I \vert})$ and any $t_0\in [0,T)$, if $\psi\in C^1((t_0,T]\times [0,1],\mathbb{R}^{\vert I \vert})$ satisfies \eqref{hjb} on $(t_0,T]\times (0,1)\times I$, then $\Vert \psi \Vert_\infty< M$. 
\end{lemma}
\begin{proof}
Let us define $M:=\Vert g\Vert _\infty +T(\Vert c\Vert _\infty +  \vert I \vert H(\delta) )+1$ and, for any $t \in [t_0,T]$,
\begin{equation*}
    \begin{array}{l}
         \underline{u}(t):=-\Vert g\Vert _\infty -(T-t)(\Vert c\Vert _\infty +  \vert I \vert H(\delta) ), \\
         \bar{u}(t):=\Vert g\Vert _\infty +(T-t)(\Vert c\Vert _\infty  +  \vert I \vert H(\delta)).
    \end{array}
\end{equation*}
We note that $\Vert \underline{u}\Vert_\infty<M$ and $\Vert \bar{u}\Vert_\infty<M$.
One has, for any $(t,s,i)\in (t_0,T]\times [0,1]\times I$,
\begin{equation}\label{partial_ineq_u}
    \begin{array}{cc}
   -\partial_t
\underline{u}(t)  \leq - \sum_{j\in I,j\neq i}H\bigg(\int_t^T(\lambda_i-\lambda_j)(r)dr\bigg) -c_i(t,s) \vspace{0.2cm} \\
      -\partial_t
\bar{u}(t) \geq -\sum_{j\in I,j\neq i}H\bigg(\int_t^T(\lambda_i-\lambda_j)(r)dr\bigg)+ c_i(t,s).
    \end{array}
\end{equation}
We will show that $\psi$ is bounded by $\underline{u}$ and $\bar{u}$.
For any $t\in (t_0,T]$, we define:
\begin{equation*}
    \gamma(t):=\max_{(x,i)\in [0,1]\times I}\underline{u}(t)-\psi_i(t,x)
  \quad  \mbox{ and } \quad
    (x_t,i_t)\in \underset{(x,i)\in [0,1]\times I}{\argmax} \underline{u}(t)-\psi_i(t,x).
\end{equation*}
Since $\underline{u}$ is independent of $x$ and $i$, we have $(x_t,i_t)\in \argmin_{(x,i)\in [0,1]\times I}\psi_i(t,x)$.
On the one hand, if $x_t\in \{0,1\}$, then $b_{i_t}(x_t)=0$. On the other hand, if $x_t\in (0,1)$, then $\partial_s\psi_{i_t}(t,x_t)=0$. Therefore, one has $b_{i_t}(x_t)\partial_s(\underline{u}(t)-\psi_{i_t}(t,x_t))= 0$.
Since $\underline{u}$ and $\psi$ are Lipschitz continuous in time uniformly in $(s,i)$, $\gamma$ is also Lipschitz continuous and thus differentiable a.e. on $(t_0,T]$.  Using the Envelop Theorem \cite[Theorem 1]{milgrom2002envelope}, $\gamma$ is absolutely continuous on $(t_0,T]$ and, for a.e. $t\in(t_0,T]$, inequality \eqref{partial_ineq_u} and equality \eqref{hjb} give:
\begin{equation*}
    \gamma'(t)=\partial_t(\underline{u}(t)-\psi_{i_t}(t,x_t))\geq \sum_{j\in I,j\neq i_t} H\bigg(\int_t^T(\lambda_i-\lambda_j)(r)dr\bigg)-H\bigg((\psi_i-\psi_j)(t, x_t)+\int_t^T(\lambda_i-\lambda_j)(r)dr\bigg).
\end{equation*}
Since for any $j\in I$ we have $\psi_{i_t}(t,x_t)\leq \psi_{j}(t,x_t)$ and  $\underline{u}_{i_t}(t,x_t)= \underline{u}_{j}(t,x_t)$, the fact that $H$ is non decreasing implies:
\begin{equation*}
 H\bigg(\psi_{i_t}(t,x_t) -\psi_{j}(t,x_t)+\int_t^T(\lambda_{i_t}-\lambda_j)(r)dr\bigg) \leq H\bigg(\int_t^T(\lambda_{i_t}-\lambda_j)(r)dr\bigg),
\end{equation*}
and thus, $\gamma'(t)\geq 0$. Since $\gamma(T)<0$, we deduce that $\psi>\underline{u}$ on $(t_0,T]\times [0,1]\times I$. With  similar arguments, one obtains $\bar{u}>\psi$ on $(t_0,T]\times [0,1]\times I$. Therefore, according to the definition of $\underline{u}$
 and $\bar{u}$, we have, for any $(t,s,i)\in (t_0,T]\times [0,1]\times I:$
 \begin{equation}
     -M<\psi_i(t,s)<M.
 \end{equation}
\end{proof}

Proposition \ref{classical_sol} will be proved by a fixed point argument \cite[Theorem 5.7]{brezis2011functional}. For this, we need several lemmas.
We fix the constant $M>0$ associated with $\delta>0$ from Lemma \ref{a_priori_estimate}. The constant $\kappa>0$ is defined below and depends only on $M,\Vert b'\Vert_\infty,T$ and $\vert I\vert$.
We consider the space $$C^{0,1}([0,T]\times [0,1],\mathbb{R}^{\vert I \vert}):=\bigg\{f\in C^0([0,T]\times [0,1],\mathbb{R}^{\vert I \vert})\,\vert \, \partial_sf_i\in C^0([0,T]\times [0,1],\mathbb{R}^{\vert I \vert})\bigg\},$$ 
endowed with the norm $\Vert \cdot \Vert_{0,1}^\kappa$ defined by:
    \begin{equation*}
        \Vert f \Vert_{0,1}^\kappa:=
        \Vert f \Vert_{\infty}^\kappa
        +
        \Vert \partial_s f \Vert_{\infty}^\kappa ,
          \end{equation*}
where, for any $h\in C^{0}([0,T]\times [0,1]\times I)$,
\begin{equation*}
    \Vert h \Vert_{\infty}^\kappa:= \sup_{(t,s,i)\in [0,T]\times [0,1]\times I}\vert h_i(t,s)e^{-\kappa(T-t)}\vert.
\end{equation*}
The space $(C^{0,1}([0,T]\times [0,1],\mathbb{R}^{\vert I \vert}), \Vert \cdot \Vert_{0,1}^\kappa)$ is a Banach space.
We fix $C_0>0$ and $C_1>0$ to be chosen below, where $C_0$ depends on $\Vert \partial_s c\Vert_\infty,\Vert  g'\Vert_\infty, \Vert b'\Vert_\infty$ and $C_1$ depends on
$M,\delta,T$ and $\vert I \vert$.
We look for a solution in the space $\Sigma$ defined by:
\begin{equation*}
    \Sigma:
    =\bigg\{f\in C^{0,1}([0,T]\times [0,1],\mathbb{R}^{\vert I \vert})\,\big\vert \,\Vert f\Vert_\infty\leq M+1\mbox{ and }\Vert \partial_s f(t,\cdot)\Vert_\infty\leq C_0e^{C_1(T-t)} \,\forall t\in [0,T]\bigg\}.
\end{equation*}
\begin{remark}
\label{existence_gamma}
The set $\Sigma$ is bounded and closed and therefore complete w.r.t. the topology induced by the norm $\Vert \cdot \Vert_{0,1}^\kappa$.
\end{remark}
In this subsection, we are looking for a solution of \eqref{hjb} as a fixed point of the map $\Gamma^\lambda:=(\Gamma^\lambda_i)_{i\in I}$, which is defined on $C^{0,1}([0,T]\times [0,1],\mathbb{R}^{\vert I \vert})$ by
\begin{equation}
\label{def_gamma}
\Gamma_i^\lambda(\phi)(t,s):=\int_t^T\sum_{j\in I,j\neq i}-H^\lambda(i,j,t,\tau,s,\phi) + c_i(\tau,S_i^{t,s}(\tau)) d\tau  + g_i(S_i^{t,s}(T)).
\end{equation}
Since the function $H$ is not necessarily bounded on $\mathbb{R}$, we need to introduce a smooth truncation $F\in C^1( \mathbb{R}, [-M-1,\,M+1])$ to obtain a fixed point in $\Sigma$. The function $F$ is defined on $\mathbb{R}$ such that $F'\geq 0$, $\Vert F'\Vert_\infty \leq 1$ and:
\begin{equation}
\label{def_f}
F(x):=\left\{
\begin{array}{ll}
-M-1 &\mbox{if }x\leq -(M+2),\\
x & \mbox{if }-M-1/2\leq x \leq M+1/2,\\
M +1& \mbox{if }M+2\leq x.
\end{array}
\right.
\end{equation}
Finally, we define the function $\Pi^\lambda$ by:
\begin{equation*}
\forall \phi \in \Sigma,\,\Pi^\lambda(\phi):=(\Pi^\lambda_1(\phi),\ldots,\Pi^\lambda_{\vert I \vert}(\phi)) \mbox{ where }\Pi^\lambda_i(\phi):=(F\circ \Gamma_i^\lambda)(\phi)\quad \forall i\in I.
\end{equation*}

The following lemma states that $\Pi^\lambda$ maps $\Sigma$ into itself.
\begin{lemma}
\label{impage_space}
For a suitable choice of the constants $C_0$ and $C_1$, one has $\Pi^\lambda(\phi)\in \Sigma$ for any $\phi\in \Sigma$.
\end{lemma}
\begin{proof}
Let $\phi\in \Sigma$ and,  for any $i\in I$, $\sigma_i:=\Gamma^\lambda_i(\phi)$. We have $\sigma\in C^{0,1}([0,T]\times [0,1],\mathbb{R}^{\vert I \vert})$.
We need to show that, for any $t\in [0,T]$,  $\Vert \partial_s \sigma(t,\cdot)\Vert_\infty\leq C_0e^{C_1(T-t)} $.
According to the definition of $\Gamma^\lambda_i$, we have, for any $(i,t)\in I\times [0,T]$,
\begin{equation*}
\begin{array}{ll}
    \Vert \partial_s \sigma_i(t,\cdot)\Vert_\infty & \leq \Vert  \partial_s S \Vert_\infty K\int_t^T \sum_{i,j\in I} \Vert \partial_s(\phi_j-\phi_i)(t,\cdot)\Vert_\infty  dt + T \Vert \partial_s S\Vert_\infty \Vert \partial_s c \Vert_\infty + \Vert \partial_s S\Vert_\infty \Vert g' \Vert_\infty \\
     & \leq \frac{C_0C}{C_1}e^{C_1(T-t)}+C,
\end{array}
\end{equation*}
where $K :=\sup_{x\in [-2M-\delta , 2M+\delta]}\vert H'(x)\vert$ and $C>0$ is a positive constant which depends on $\Vert \partial_s S \Vert_\infty, \Vert\partial_s c \Vert_\infty$, $\Vert g' \Vert_\infty$, $\vert I \vert,T$ and $K$.
Choosing carefully $C_0$ and $C_1$ depending on $C$, one obtains that, for any $(i,t)\in I\times [0,T]$,
$$\Vert \partial_s \sigma_i(t,\cdot)\Vert_\infty \leq C_0e^{C_1(T-t)}.$$
Finally, from the definition of $F$ in \eqref{def_f}, the function $\Pi(\phi)=(F(\sigma_1),\ldots,F(\sigma_{\vert I\vert}))$ is in $\Sigma$.
\end{proof}
The existence of a fixed point of $\Pi^\lambda$ in $\Sigma$ is established in the following lemma.
\begin{lemma}
\label{contraction_lemma}
\label{fixed_point_pi}
For a suitable choice of the constant $\kappa$, the function $\Pi^\lambda$ admits a unique fixed point $\psi^\lambda\in \Sigma$. In addition, we have $\psi^\lambda\in C^1([0,T]\times [0,1],\mathbb{R}^{\vert I \vert})$.
\end{lemma}

\begin{proof}
Let $\phi^1,\phi^2\in \Sigma$. Since $\phi^1$ and $\phi^2$ are bounded by $M+1$, one has, for any $(t,s,i)\in [0,T]\times [0,1]\times I$,
\begin{equation}\label{ineg_infty_gamma}
\begin{array}{ll}
\vert \Gamma_i^\lambda(\phi^1)(t,s)-  \Gamma_i^\lambda(\phi^2)(t,s)\vert 
& \leq \sum_{j\in I,j\neq i}\int_t^T \vert H^\lambda(i,j,t,\tau,s,\phi^1)-H^\lambda(i,j,t,\tau,s,\phi^2)\vert \, d\tau \vspace{0.1cm} \\
  & \leq  C\sum_{j\in I} \int_t^T\vert \phi^1_j(\tau, S_i^{t,s}(\tau))
  -\phi^2_j(\tau, S_i^{t,s}(\tau))\vert \,d\tau\\
 & \leq \Vert \phi^1 -\phi^2\Vert_\infty^\kappa \frac{Ce^{\kappa(T-t)}}{\kappa},
\end{array}
\end{equation}
where $C>0$ is a constant depending on $\delta,\vert I\vert$ and $\sup_{x\in [-2(M+1)-\delta , 2(M+1)+\delta]}\vert H'(x)\vert$.
Since $H'$ is Lipschitz continuous
on $[-2 (M+1),\,2(M+1)]$, one has, for any $(t,s,i)\in [0,T]\times [0,1]\times I$,
\begin{equation}
\label{local_deriv_gamma}
\begin{array}{l}
   \big\vert \partial_s\big( \Gamma_i^\lambda(\phi^1)(t,s)-  \Gamma_i^\lambda(\phi^2)(t,s)\big)\big\vert   \\
   \leq \sum_{j\in I,j\neq i}\int_t^T \left\vert \partial_s (\phi^1_i-\phi^1_j)(\tau,S_i^{t,s}(\tau))H'\Big((\phi^1_i-\phi^1_j)\big(\tau,S_i^{t,s}(\tau)\big)+\int_\tau^T(\lambda_i-\lambda_j)(r)dr\Big) \right .
   \\   \left.\hspace{2.2cm}- \partial_s (\phi^2_i-\phi^2_j)(\tau,S_i^{t,s}(\tau))H'\Big((\phi^2_i-\phi^2_j)\big(\tau,S_i^{t,s}(\tau)\big)+\int_\tau^T(\lambda_i-\lambda_j)(r)dr\Big)\right\vert d\tau \\
      \leq \Vert  \phi^1 - \phi^2\Vert_{0,1}^\kappa \frac{Ke^{\kappa(T-t)}}{\kappa},
\end{array}
\end{equation}
where $K>0$ is a constant that depends on $\vert I\vert, T, \Vert b' \Vert_\infty,M,C_0,C_1$ and on the bound and the Lipschitz constant of $H'$ on $[-2 (M+1)-\delta,\,2(M+1)+\delta]$. From \eqref{ineg_infty_gamma} and \eqref{local_deriv_gamma}, one obtains that:
$$ \Vert \Gamma^\lambda(\phi^1)-\Gamma^\lambda(\phi^2)\Vert_{0,1}^\kappa \leq \max(C/ \kappa,K/ \kappa)\Vert \phi^1 -\phi^2\Vert_{0,1}^\kappa.$$ 
Choosing $\kappa>\max(C,K)$ and the function $F$ being non-expansive, one deduces that the function $\Pi^\lambda$ is a contraction on $\Sigma$ and that it admits a unique fixed point $\psi^\lambda\in \Sigma$.  By the definitions of $\Pi^\lambda$ and $\Sigma$, it is straightforward that $\psi^\lambda\in C^1([0,T]\times [0,1],\mathbb{R}^{\vert I \vert})$.
\end{proof}

We now turn to the proof of Proposition \ref{classical_sol}.
\begin{proof}[Proof of Proposition \ref{classical_sol}]

The first step of the proof is to show that $\psi^\lambda$, the fixed point of $\Pi^\lambda$ in $\Sigma$, is also a fixed point of $\Gamma^\lambda$ in $\Sigma$. Tho this end, one needs to prove that the truncation $F$ has no effect on $\psi^\lambda$, which is equivalent to showing that $\Vert \psi^\lambda \Vert_\infty\leq M+1/2$.
Let $t_0\in[0,T]$ be the minimum time such that $\Vert\psi^\lambda(t,\cdot)\Vert_\infty\leq M+1/2$ for any $t\in [t_0,T]$. Since $\Vert \psi_i^\lambda(T,\cdot)\Vert_\infty = \Vert g_i\Vert_\infty<M$ for any $i\in I$ and $\psi^\lambda$ is continuous, the time $t_0$ is smaller than $T$. The function $\psi^\lambda$ is a fixed point of $\Gamma^\lambda$ on $[t_0,T]$ and thus, $\psi^\lambda$ is a solution of \eqref{def_integral_equation_2} on $[t_0,T]\times [0,1]\times I$. By Lemma \ref{eq_int_pde}, we deduce that $\psi^\lambda$ satisfies \eqref{hjb} on $(t_0,T]\times [0,1]\times I$. If $t_0=0$, then the conclusion follows. If $t_0>0$, then $\Vert\psi^\lambda(t_0,\cdot)\Vert_\infty=M+1/2$. By Lemma \ref{a_priori_estimate}, one also has $\Vert\psi^\lambda(t_0,\cdot)\Vert_\infty\leq M$. Hence, there is a contradiction.
Therefore, $\psi^\lambda$ is a fixed point in $\Sigma$ of $\Gamma^\lambda$. Since any fixed point of $\Gamma^\lambda$ is a fixed point of $\Pi^\lambda$ and that by Lemma \ref{fixed_point_pi}, $\Pi^\lambda$ has a unique fixed point in $\Sigma$, one deduces that $\psi^\lambda$ is the unique solution of \eqref{def_integral_equation_2}. By Lemma \ref{eq_int_pde},  $\psi^\lambda$ is the unique classical solution of \eqref{hjb} on $[0,T]\times [0,1]\times I$
\end{proof}
\begin{remark}\label{bound_derivative}
One can show that $\Vert \partial_t \psi^\lambda \Vert_\infty+ \Vert \partial_s \psi^\lambda \Vert_\infty$ is bounded by a function that is non-decreasing w.r.t. the variable $\delta$. Indeed, by Lemma \ref{impage_space} and its proof, $\Vert \partial_s \psi^\lambda\Vert_\infty$ is bounded by $C_0e^{C_1T}$ where $C_0$ and $C_1$ depend on $M$ and on $\delta$ by the definition of $M$ in Lemma \ref{a_priori_estimate}. By Lemma \ref{eq_int_pde}, $\psi^\lambda$ is a classical solution of \eqref{hjb}. Therefore,
$$\Vert \partial_t\psi^\lambda\Vert_\infty \leq \Vert b\Vert_\infty \Vert \partial_s \psi^\lambda \Vert_\infty + \Vert c\Vert_\infty + \sup_{x\in [-2M-\delta , 2M+\delta]}\vert H(x)\vert.$$
Thus, $\Vert \partial_t\psi^\lambda\Vert_\infty $ is bounded by a non-decreasing function of $\delta$.
\end{remark}

\subsection{The Hamilton-Jacobi equation for measure valued data}\label{lambda_measure_valued_data}
In this subsection, we prove the following
\begin{proposition} \label{existence_measure_fixed_point}
For any $\lambda\in\mathcal{M}^+([0,T],\mathbb{R}_+^{\vert I \vert})$, there exists a unique $\psi^\lambda\in\Lip([0,T]\times [0,1],\mathbb{R}^{\vert I \vert})$ satisfying that, for any $(t,s,i)\in [0,T]\times [0,1]\times I$,
\begin{equation}\label{int_form}
    \psi^\lambda_i(t,s)=\int_t^T\sum_{j\in I,j\neq i}-H^\lambda(i,j,t,\tau,s,\psi^\lambda) + c_i(\tau,S_i^{t,s}(\tau)) d\tau  + g_i(S_i^{t,s}(T)).
\end{equation}
In addition, the map $\lambda\mapsto \psi^\lambda$ is continuous from $\mathcal{M}^+([0,T],\mathbb{R}_+^{\vert I \vert})$, endowed with the weak$^\ast$ topology, to $C^0([0,T]\times [0,1],\mathbb{R}^{\vert  I\vert})$, endowed with the norm $\Vert \cdot \Vert_\infty$.
\end{proposition}
We recall that the definition of $H^\lambda$, with $\lambda\in \mathcal{M}^+([0,T],\mathbb{R}_+^{\vert I \vert})$, is given in \eqref{def_h_lambda_measure}.
The proof of Proposition \ref{existence_measure_fixed_point} relies on the results of the previous subsection. We define the map $\Theta:C_\delta^0([0,T],\mathbb{R}_+^{\vert I \vert})\to C^1([0,T]\times [0,1],\mathbb{R}^{\vert I \vert})$ by:
\begin{equation}
\label{def_theta}
    \Theta:\lambda\mapsto \psi^\lambda,
\end{equation}
where $\psi^\lambda$ is given by Proposition \ref{classical_sol}. We know from Proposition \ref{classical_sol} that $\Theta$ is well defined on $C^0([0,T],\mathbb{R}_+^{\vert I \vert})$. 
We want to show that $\Theta$ can be continuously extended to a function defined on $\mathcal{M}_{\delta}^+([0,T],\mathbb{R}_+^{\vert I \vert })$ with values in $\Lip([0,T]\times [0,1],\mathbb{R}^{\vert I \vert})$.
We define the distance $\mathcal{D}$ on $\mathcal{M}^+([0,T],\mathbb{R}_+^{\vert I \vert})^2$ by
\begin{equation}\label{def_metric_D}
    \mathcal{D}(\lambda,\mu):=    \int_0^T\sum_{i\in I}\left\vert \int_t^T(\lambda_i-\mu_i)(d\tau)\right\vert dt+ \sum_{i\in I}\left\vert  \int_0^T(\lambda_i-\mu_i)(dt)\right\vert .
\end{equation}
We show in Section \ref{prop_metric_d} in Appendix \ref{appendices_section}, that $\mathcal{D}$ is a distance and we provide some properties on this metric.
\begin{lemma}\label{cont_uniform}
The map $\Theta$ can be extended to a Lipschitz continuous map from $\mathcal{M}_{\delta}^+([0,T],\mathbb{R}_+^{\vert I \vert })$, endowed with distance $\mathcal{D}$, to $C^0([0,T]\times [0,1],\mathbb{R}^{\vert I \vert})$, endowed with the norm $\Vert \cdot \Vert_\infty$. In addition we have $\Theta(\lambda)\in \Lip([0,T]\times [0,1],\mathbb{R}^{\vert I \vert})$ for any $\lambda\in\mathcal{M}_{\delta}^+([0,T],\mathbb{R}_+^{\vert I \vert })$.
\end{lemma}

\begin{proof}
We need to show that there exists a constant $C>0$ such that, for any $\lambda^1,\lambda^2\in C^0_\delta([0,T],\mathbb{R}_+^{\vert I \vert})$, we have: $\Vert \Theta(\lambda^1)-\Theta(\lambda^2)\Vert_\infty\leq C\mathcal{D}(\lambda^1,\lambda^2)$. Since $H$ is locally Lipschitz, there exists a constant $K>0$ such that, for any $(t,s,i)\in [0,T]\times [0,1]\times I$,
\begin{equation*}
    \begin{array}{ll}
     \vert \Theta(\lambda^1)_i(t,s)- \Theta(\lambda^2)_i(t,s) \vert & =  \vert \psi^{\lambda^1}_i(t,s)- \psi^{\lambda^2}_i(t,s)\vert  \vspace{0.2cm}   \\
         &  =  \left\vert \sum_{j\in I,j\neq i}\int^T_t  
      H^{\lambda^1}(i,j,t,\tau,s,\psi^{\lambda^1})-  H^{\lambda^2}(i,j,t,\tau,s,\psi^{\lambda^2})d\tau\right\vert \vspace{0.2cm}
       \\
         & \leq \vert I \vert\sum_{j\in I}K\int^T_t \bigg(
         \vert \psi_j^{\lambda^1}(\tau,S_i^{t,s}(\tau))-\psi_j^{\lambda^2}(\tau,S_i^{t,s}(\tau))\vert +\vert \int_\tau^T\lambda^1_j(r)dr - \int_\tau^T\lambda^2_j(r)dr\vert  \bigg)d\tau.
    \end{array}
\end{equation*}
Taking the supremum over $I\times [0,1]$ yields:
\begin{equation*}
\begin{array}{l}
  \Vert \psi^{\lambda^1}(t,\cdot)- \psi^{\lambda^2}(t,\cdot)\Vert_\infty
    \\\leq \vert I \vert^2  \int^T_t 
         K\Vert \psi^{\lambda^1}(\tau,\cdot)-\psi^{\lambda^2}(\tau, \cdot)\Vert_\infty d\tau + \vert I \vert^2K\underset{i\in I}{\sup}\int_t^T\bigg\vert \int_\tau^T\lambda^1_i(r)dr - \int_\tau^T\lambda^2_i(r)dr\bigg\vert d\tau \vspace{0.2cm}\\
         \leq \vert I \vert^2  \int^T_t 
         K\Vert \psi^{\lambda^1}(\tau,\cdot)-\psi^{\lambda^2}(\tau, \cdot)\Vert_\infty d\tau + \vert I \vert^2K\mathcal{D}({\lambda^1},{\lambda^2}).
\end{array}
\end{equation*}
Then, by applying Gronwall Lemma to $t\mapsto \Vert \psi^{\lambda^1}(t,\cdot)-\psi^{\lambda^2}(t, \cdot)\Vert_\infty$, one has, for any $t\in [0,T]$,
\begin{equation}
\label{diff_mu_dist_d}
    \Vert \psi^{\lambda^1}(t,\cdot)- \psi^{\lambda^2}(t,\cdot)\Vert_\infty  \leq C\mathcal{D}({\lambda^1},{\lambda^2}),
\end{equation}
where the constant $C>0$ depends on $\delta,K,T$ and $\vert I \vert $. Therefore,
\begin{equation*}
  \Vert \Theta({\lambda^1})- \Theta({\lambda^2})\Vert_\infty =\Vert \psi^{\lambda^1}- \psi^{\lambda^2}\Vert_\infty \leq C\mathcal{D}(\lambda^1, \lambda^2).  
\end{equation*}
From the previous inequality and Lemma \ref{density_c_infty} in Section \ref{prop_metric_d}, the map $\Theta$ can be continuously extended to a Lipschitz continuous map from $\mathcal{M}_{\delta}^+([0,T],\mathbb{R}_+^{\vert I \vert })$ to $C^0([0,T]\times [0,1],\mathbb{R}^{\vert I \vert})$. 

Finally, by Lemma \ref{density_c_infty} we approximate $\lambda\in \mathcal{M}_{\delta}^+([0,T],\mathbb{R}_+^{\vert I \vert })$ by a sequence in $\{\lambda^n\}_n$ in $C^0_\delta([0,T],\mathbb{R}_+^{\vert I \vert})$ w.r.t. the distance $\mathcal{D}$. One has that $\{\Theta(\lambda^n)\}_n$ is uniformly Lipschitz continuous on $[0,T]\times [0,1]\times I$ according to Remark \ref{bound_derivative}. Thus, $\Theta(\lambda)$ is in $\Lip([0,T]\times [0,1],\mathbb{R}^{\vert I \vert})$.
\end{proof}

\begin{proof}[Proof of Proposition \ref{existence_measure_fixed_point}]
Let $\lambda\in \mathcal{M}^+([0,T],\mathbb{R}_+^{\vert I \vert})$ and $\delta>0$ be such that $\sum_{i\in I}\lambda_i([0,T])<\delta$. By Lemma \ref{cont_uniform}, the map $\Theta$ is Lipschitz continuous from $\mathcal{M}^+([0,T],\mathbb{R}_+^{\vert I \vert})$, endowed with the distance $\mathcal{D}$, to $\Lip([0,T]\times [0,1],\mathbb{R}^{\vert I \vert})$ endowed with the norm $\Vert \cdot \Vert_\infty$. We set: $\psi^\lambda:=\Theta(\lambda)$. We want to show that $\psi^\lambda$ is the unique solution of \eqref{int_form}. By Lemma \ref{density_c_infty} in Section \ref{prop_metric_d}, we consider a sequence $\{\lambda^n\}_n$ in $C_\delta^0([0,T],\mathbb{R}_+^{\vert I \vert})$ converging to $\lambda$ w.r.t. the distance $\mathcal{D}$. For any $n\in \mathbb{N}$, we consider $\psi^{\lambda^n}:=\Theta(\lambda^n)$, the solution of \eqref{def_integral_equation_2} associated with $\lambda^n$.
By the continuity of $\Theta$ stated in Lemma \ref{cont_uniform}, the sequence $\{\psi^{\lambda^n}\}_n$ converges to $\psi^\lambda$ w.r.t. the norm $\Vert \cdot \Vert_\infty$ and by the continuity of $H$ and the dominated convergence theorem, $\psi^\lambda$ is a solution of \eqref{int_form} associated with $\lambda$. 
Now we want to show that $\psi^\lambda$ is the unique  solution of \eqref{int_form} associated with $\lambda$. Assume the function $\bar{\psi}$ is a solution of \eqref{int_form} associated with $\lambda$.
Then, by applying similar computations as in the proof of Lemma \ref{cont_uniform}, one can show that:
\begin{equation*}
    \Vert \psi^\lambda - \bar{\psi} \Vert_\infty \leq C\mathcal{D}(\lambda,\lambda)=0.
\end{equation*}
Thus, $\psi^\lambda$ is the unique  solution of \eqref{int_form} associated with $\lambda$.
The continuity of the map $\lambda\mapsto \psi^\lambda$ from $\mathcal{M}^+([0,T],\mathbb{R}_+^{\vert I \vert})$, endowed with the weak$^\ast$ topology, to $C^0([0,T]\times [0,1]\times  I)$, endowed with the norm $\Vert \cdot \Vert_\infty$, is a consequence of the continuity of $\Theta$ stated in Lemma \ref{cont_uniform}.
\end{proof}

\begin{remark} \label{control_psi_norm_lambda}
By Remark \ref{bound_derivative} and Lemma \ref{cont_uniform}, one can show that for any $\lambda \in \mathcal{M}^+([0,T],\mathbb{R}_+^{\vert I \vert})$,  the solution $\psi^\lambda$ of \eqref{int_form} is such that,
$\max(\Vert \psi^\lambda\Vert_\infty, \Vert \partial_s\psi^\lambda\Vert_\infty)$ depends only on $H, b,c,g$ and $\sum_{i \in I}\lambda_i([0,T])$.
\end{remark}

\subsection{Analysis of weak solution of the Hamilton-Jacobi equation (2.5)}\label{existence_measure_weak_solution}
For any $\lambda\in \mathcal{M}^+([0,T],\mathbb{R}_+^{\vert I \vert})$, we consider the function $L^\lambda$ defined for any $i\in I$ and any $ t \in [0,T]$ by
\begin{equation}
    \label{def_L}
    L^\lambda_i(t):=\lambda_i([t,T]).
\end{equation}
Since, for $i\in I$, $\lambda_i$ is a finite measure on $[0,T]$, the function $L^\lambda$ is in $L^1([0,T],\mathbb{R}^{\vert I \vert})$. Observing that, for any test function compactly supported $f\in C^1_c((0,T),\mathbb{R}^{\vert I \vert})$, one has for any $i\in I$:
\begin{equation*}
    \int_0^T\partial_tf_i(t)L^\lambda_i(t)dt=
    \int_0^T\partial_tf_i(t)\Big(\int_t^T  \lambda(dr)\Big)dt=\int_0^T\Big(\int_0^r\partial_tf_i(t)dt\Big)\lambda_i(dr)=\int_0^Tf_i(t){\color{black}\lambda_i(dt)},
\end{equation*}
thus, $L^\lambda$ is in $BV([0,T],\mathbb{R}^{\vert I \vert})$ \cite[Definition 3.1]{ambrosio2000functions}.
We introduce the notion of weak solution for equation \eqref{hjb_0}.
\begin{definition}
\label{def_weak_sol_hjb}
For a given $\lambda\in \mathcal{M}^+([0,T],\mathbb{R}_+^{\vert I \vert})$, a function $\varphi\in BV((0,T)\times (0,1),\mathbb{R}^{\vert I \vert})$ is a weak solution of equation \eqref{hjb_0} if $\varphi - L^\lambda$ is in $ \Lip([0,T]\times [0,1],\mathbb{R}^{\vert I \vert})$ and if, for any test function $f\in C^1([0,T]\times [0,1],\mathbb{R}^{\vert I \vert})$,
\begin{equation}
\label{def_sub_so_2}
\begin{array}{l}
\int_0^1\varphi_i(0,s)f_i(0,s)ds -\int_0^1g_i(s)f_i(T,s)ds
+\int_0^T\int_0^1\left(\partial_tf_i(t,s)+\partial_s(f_i(t,s)b_i(s))\right)\varphi_i(t,s)dsdt \vspace{0.1cm} \\
+\int_0^T\int_0^1
\bigg(\sum_{j\in I,j\neq i}H(\varphi_i(t,s)-\varphi_j(t,s))-c_i(t,s)\bigg)f_i(t,s)dtds-\sum_{i\in I}\int_0^T\int_0^1f_i(t,s)ds\lambda_i(dt)\vspace{0.1cm}\\
= 0,
\end{array}
\end{equation}
where $\varphi_i(0,\cdot)$ is understood in the sense of trace \cite[Theorem 3.87]{ambrosio2000functions}.
\end{definition}
\begin{remark}\label{multiple_remarks}
\begin{enumerate}
    \item There is no boundary condition in \eqref{def_sub_so_2}. This is due to the fact that $b(0)=b(1)=0$, involving a null incoming flow in the domain $[0,1]$.
    \item Since $\varphi$ is in $BV([0,T],\mathbb{R}^{\vert I \vert})+\Lip([0,T]\times [0,1],\mathbb{R}^{\vert I \vert})$, $\varphi(0,\cdot)$ exists in the sense of trace. In addition, $\varphi$ is bounded and thus, it belongs to $L^2([0,T]\times [0,1],\mathbb{R}^{\vert I \vert})$ so that the integrals in \eqref{def_sub_so_2} exist.
    \item \label{finale_trace}
    The term $\lambda_i(\{T\})$ in the final condition of \eqref{hjb_0} is present in \eqref{def_sub_so_2} through the term $\textstyle \int_0^T\int_0^1f_i(t,s)ds\lambda_i(dt)$.
\end{enumerate}
\end{remark}
The main result of this subsection is the following.
\begin{theorem}\label{eq_weak_rep_form}
Let $\lambda\in \mathcal{M}^+([0,T],\mathbb{R}_+^{\vert I \vert})$. A function $\varphi\in  \Lip([0,T]\times [0,1],\mathbb{R}^{\vert I \vert}) + BV([0,T],\mathbb{R}^{\vert I \vert})$ is the unique weak solution of \eqref{hjb_0} in the sense of Definition \ref{def_weak_sol_hjb}, if and only if it satisfies that, for any $(t,s,i)\in [0,T]\times [0,1]\times I$,
\begin{equation}\label{int_form_varphi}
    \varphi_i(t,s)=\int_t^T\sum_{j\in I,j\neq i}-H\Big((\varphi_i-\varphi_j)(\tau, S_i^{t,s}(\tau))\Big)+c_i(\tau,S_i^{t,s}(\tau))  d\tau +L^\lambda_i(t)+ g_i(S_i^{t,s}(T)).
\end{equation}
\end{theorem}
Theorem \ref{eq_weak_rep_form} is a consequence of Lemmas \ref{Int_sol_implies_weak_sol} and \ref{int_sol_imp_weak_sol}. Lemma  \ref{Int_sol_implies_weak_sol} shows that a function satisfying \eqref{int_form_varphi} is a weak solution of \eqref{hjb_0}. Lemma  \ref{int_sol_imp_weak_sol} shows the converse.
In what follows, for any $\lambda\in \mathcal{M}^+([0,T],\mathbb{R}_+^{\vert I \vert})$, we define $\psi^\lambda$ as in Proposition \ref{existence_measure_fixed_point} and define $\varphi^\lambda\in \Lip([0,T]\times [0,1],\mathbb{R}^{\vert I \vert})+ BV([0,T],\mathbb{R}^{\vert I \vert})$ by: 
\begin{equation}\label{def_varphi}
    \varphi^\lambda:=\psi^\lambda + L^\lambda.
\end{equation}
\begin{remark}\label{link_psi_varphi}
For any $\lambda\in \mathcal{M}^+([0,T],\mathbb{R}_+^{\vert I \vert})$, $\varphi^\lambda$ is a solution of \eqref{int_form_varphi} on $[0,T]\times [0,1]\times I$ if and only if $\psi^\lambda$ is a solution of \eqref{int_form} on $[0,T]\times [0,1]\times I$. According to Remark \ref{control_psi_norm_lambda}, the quantity $\max(\Vert \varphi^\lambda\Vert_\infty,\Vert \partial_s\varphi^\lambda\Vert_\infty)$ depends on $\sum_{i\in I}\lambda_i([0,T])$.
\end{remark}

\begin{remark}\label{rem_classical_sol_hjb}
    Following Remark \eqref{link_psi_varphi}, for any $\lambda\in C^0([0,T],\mathbb{R}_+^{\vert I \vert})$, $\varphi^\lambda$ is a classical solution of \eqref{hjb_0} (in the sense that $\varphi^\lambda$ is in $C^1([0,T]\times [0,1],\mathbb{R}^{\vert I \vert})$ and satisfies \eqref{hjb_0} for any $(t,s,i)\in [0,T]\times [0,1]\times I)$) if and only if $\psi^\lambda$ is a classical solution of \eqref{hjb}.
\end{remark}

\begin{remark}\label{note_on_final_condition}
     For a $\lambda$ in $C^0([0,T],\mathbb{R}_+^{\vert I \vert})$, the final condition in \eqref{hjb_0} is understood as $\varphi^\lambda_i(T,s)=g_i(s)$. For a $\lambda$ in $\mathcal{M}^+([0,T],\mathbb{R}_+^{\vert I \vert})$, by the definition of $L^\lambda$ in \eqref{def_L}, one has $\varphi^\lambda_i(T,s)=g_i(s)+\lambda_i(\{T\})$, similarly to the final condition of the Hamilton Jacobi equation \eqref{hjb_0}.
\end{remark}

\begin{lemma}\label{approx_lambda_smooth}
For any $\lambda\in \mathcal{M}^+([0,T],\mathbb{R}_+^{\vert I \vert})$, there exists a sequence $\{(\lambda^n,\varphi^n)\}_n$ such that
\begin{enumerate}[label=(\roman*)]
    \item $\{\lambda^n\}_n$ is in $C^0([0,T],\mathbb{R}_+^{\vert I \vert})$ and converges to $\lambda$ w.r.t. the weak$^\ast$ topology in $\mathcal{M}^+([0,T],\mathbb{R}_+^{\vert I \vert})$,
    \item for any $n\in \mathbb{N}$, $\varphi^n$ is a classical solution of \eqref{hjb_0} associated with $\lambda^n$ on $(0,T)\times (0,1)\times I$,
    \item we have $\lim_{n\to \infty}\Vert \varphi^\lambda(t,\cdot)-\varphi^n(t,\cdot)\Vert_\infty=0$ for a.e. $ t\in [0,T]$ and $\lim_{n\to \infty}\Vert \varphi^\lambda(0,\cdot)-\varphi^n(0,\cdot)\Vert_\infty=0$.
\end{enumerate}
\end{lemma}
\begin{proof}
By Lemma \ref{density_c_infty}, we consider a sequence $\{\lambda^n\}_n$ in $C^0([0,T],\mathbb{R}_+^{\vert I \vert})$ that weakly$^\ast$ converges to $\lambda$. We set $\psi^{n}:=\psi^{\lambda^n}$, where $\psi^{\lambda^n}$ is defined by Proposition \ref{existence_measure_fixed_point}, and $\varphi^n:=\psi^n+L^{\lambda^n}$. 
By Lemma \ref{eq_int_pde}, for any $n\in \mathbbm{N}$, $\psi^n$ is a classical solution of \eqref{hjb}. Then, for any $n\in \mathbbm{N}$, $\varphi^n$ is a classical solution of \eqref{hjb_0} associated with $\lambda^n$ on $(0,T)\times (0,1)\times I$.
The weak$^\ast$ convergence of $\{\lambda^n\}_n$ implies (see Lemma \ref{rem_weak_conv_dist} in Section \ref{prop_metric_d} and its proof):
\begin{equation*}
\begin{array}{lll}
   \lim_{n\to \infty}\,L^{\lambda^n}(t)=L^\lambda(t)\mbox{ for a.e. } t\in [0,T] & \mbox{ and }
     &  \lim_{n\to \infty}\,L^{\lambda^n}(0)=L^\lambda(0).
\end{array}
\end{equation*}
By Proposition \ref{existence_measure_fixed_point}, $\{\psi^n\}_n$ converges to $\psi^\lambda$ w.r.t. the norm $\Vert \cdot \Vert_\infty$. Then, the two previous equalities imply that:
\begin{equation*}
\begin{array}{lll}
   \lim_{n\to \infty}\Vert \varphi^\lambda(t,\cdot)-\varphi^n(t,\cdot)\Vert_\infty=0\mbox{ for a.e. } t\in [0,T] & \mbox{ and }
     & \lim_{n\to \infty}\Vert \varphi^\lambda(0,\cdot)-\varphi^n(0,\cdot)\Vert_\infty=0.
\end{array}
\end{equation*}
\end{proof}

\begin{lemma}\label{Int_sol_implies_weak_sol}
For any  $\lambda\in \mathcal{M}^+([0,T],\mathbb{R}_+^{\vert I \vert})$, $\varphi^\lambda$ is a weak solution of \eqref{hjb_0} in the sense of Definition \ref{def_weak_sol_hjb}. 
\end{lemma}
\begin{proof}
We consider a sequence $\{\lambda^n,\varphi^n\}_n$ defined as in Lemma \ref{approx_lambda_smooth}. For any $n \in \mathbb{N}$, $\varphi^n$ is a classical solution of \eqref{hjb_0}. Thus, for any test function $f\in C^1([0,T]\times [0,T],\mathbb{R}^{\vert I \vert})$:
\begin{equation}
\label{weak_sol_smooth_case}
\begin{array}{l}
\int_0^1\varphi^n_i(0,s)f_i(0,s)ds -\int_0^1g_i(s)f_i(T,s)ds
+\int_0^T\int_0^1\Big(\partial_tf_i(t,s)+\partial_s(f_i(t,s)b_i(s))\Big)\varphi^n_i(t,s)dsdt \vspace{0.1cm} \\
+\int_0^T\int_0^1
\Big(\sum_{j\in I,j\neq i}H(\varphi^n_i(t,s)-\varphi^n_j(t,s))-c_i(t,s)\Big)f_i(t,s)dtds-\sum_{i\in I}\int_0^T\int_0^1f_i(t,s)ds\lambda^n_i(t)dt\vspace{0.1cm}\\
= 0.
\end{array}
\end{equation}
The conclusion follows by using Lemma \ref{approx_lambda_smooth}, the continuity of $H$, the limit of \eqref{weak_sol_smooth_case} when $n$ tends to infinity and by applying the dominated convergence theorem.
\end{proof}
The next lemma states the converse of Lemma \ref{Int_sol_implies_weak_sol}.
\begin{lemma}
\label{int_sol_imp_weak_sol}
For any $(\lambda,\varphi)\in \mathcal{M}^+([0,T],\mathbb{R}_+^{\vert I \vert})\times (\Lip([0,T]\times [0,1],\mathbb{R}^{\vert I \vert})+ BV([0,T],\mathbb{R}^{\vert I \vert}))$, if $\varphi$ is a weak solution of \eqref{hjb_0} associated with $\lambda$ in the sense of Definition \ref{def_weak_sol_hjb}, then $\varphi$ satisfies \eqref{int_form_varphi} for any $(t,s,i)\in [0,T]\times [0,1]\times I$.
\end{lemma}
The proof Lemma \ref{int_sol_imp_weak_sol} is postponed to Section \ref{proof_lemma_eq_weak_sol_int_sol}
\\

With the above lemmas, the proof of Theorem \ref{eq_weak_rep_form} is straightforward.
\begin{proof}[Proof of Theorem \ref{eq_weak_rep_form}]The proof is a direct consequence of Lemmas \ref{Int_sol_implies_weak_sol} and \ref{int_sol_imp_weak_sol}. The uniqueness of a weak solution is deduced by Remark \ref{link_psi_varphi}. Indeed, since $\psi^\lambda$ is the unique solution of \eqref{int_form}, $\varphi^\lambda$ is the unique solution of \eqref{int_form_varphi} and thus, the unique weak solution of \eqref{hjb_0} in the sense of Definition \ref{def_weak_sol_hjb}. 
\end{proof}

\section{The dual problem}\label{dual_prob_section}

This section aims at showing that there exists a convex optimization problem whose dual problem is Problem \eqref{problemE}. 
 We consider the following spaces:
\begin{equation*}
X_0 = C^1([0,T]\times [0,1],\mathbb{R}^{\vert I \vert})\times C^0([0,T],\mathbb{R}^{\vert I \vert}) \mbox{ and }X_1:= C^0([0,T]\times[0,1],\mathbb{R}^{\vert I \vert}) \times
C^0([0,T]\times [0,1],\mathbb{R}^{\vert I \vert\times \vert I \vert }).
\end{equation*}
Given $(\varphi, \lambda)\in X_0$, we consider the following inequality:
\begin{equation}
\label{hjb_ineq}
\begin{array}{ll}
- \partial_ t{\varphi}_i(t,s)-b_i(s)\partial_s{\varphi}_i(t,s)-c_i(t,s)-\lambda_i(t)+\sum_{j\in I,j\neq i}H(({\varphi}_i-{\varphi}_j)(t,s))\leq0 & \mbox{on }(0,T)\times (0,1)\times I,\\
{\varphi}_i(T,\cdot)\leq  g_i & \mbox{on } (0,1)\times I.
\end{array}
\end{equation}
The set $\mathcal{K}_0$  is defined by: $\mathcal{K}_0:=\{(\varphi,\lambda)\in X_0,\, \lambda\geq 0  \mbox{ and } \varphi\mbox{ satisfies }\eqref{hjb_ineq}\mbox{ associated with }\lambda \}$. We introduce the function ${A}$, defined on $\mathcal{K}_0$ by :
\begin{equation}
\label{def_dual_fun}
{A}(\varphi,\lambda):=
\sum_{i \in I}\int_0^1  - \varphi_i(0,s)m_i^0(ds) + \int_0^T\lambda_i(t)D_i(t)dt,
\end{equation}
and the following problem is considered:
\begin{equation}
\label{dual_prob}
\inf_{(\varphi,\lambda)\in \mathcal{K}_0}\,A(\varphi,\lambda).
\end{equation}
\begin{remark}\label{remark_fina_cond_mult}
    Since we consider multipliers $\lambda$ in $C^0([0,T],\mathbb{R}^{\vert I \vert})$ in Problem \eqref{dual_prob}, one has for any $i\in I$ that $\textstyle\int_T^T\lambda_i(t)dt{\color{black}=0}$. Thus, the final condition in \eqref{hjb_ineq} is $\varphi_i(T,\cdot)\leq g_i$, unlike the final condition $\varphi_i(T,\cdot) = g_i + \lambda_i(\{T\})$ in \eqref{hjb_0} where we consider multipliers $\lambda$ in $\mathcal{M}^+([0,T],\mathbb{R}_+^{\vert I \vert})$.
\end{remark}
The main result of this section is the following.
\begin{theorem}
\label{pb_in_duality}
We have:
\begin{equation*}
\inf_{(\varphi,\lambda)\in \mathcal{K}_0}\,A(\varphi,\lambda)=
- \underset{(m,E)\in \mathcal{S}(m^0,D)}{\min}\,\tilde{J}(m,E).
\end{equation*}
\end{theorem}
This theorem is proved further in this section. It is a direct application of the Fenchel-Rockafellar duality theorem \cite[Chapter 3]{ekeland1999convex}, which is stated
as follows.
\begin{theorem}[Fenchel-Rockafellar]Let $X_0$ and $X_1$ be Banach spaces and denote by $X_0^\ast$ and $X_1^\ast$  their duals. Suppose $\Lambda:X_0\mapsto X_1$ is a continuous linear operator with adjoint $\Lambda^\ast:X_1^\ast \mapsto X_0^\ast$. Finally, let $\mathcal{F}$ and  $\mathcal{G}$ be convex functionals on $X_0$ and $X_1$, with respective Fenchel conjugates denoted by $\pmb{\F}$ and $\pmb{\G}$. 
If 
\begin{equation*}
      \underset{x\in X_0}{\inf}\mathcal{F}(x)+\mathcal{G}(\Lambda x)
\end{equation*}
is finite and there exists $x_0\in X_0$ such that $\mathcal{F}(x_0)<+\infty$ and $\mathcal{G}$ is continuous at $\Lambda x_0$, then
 \begin{equation*}
     \underset{x\in X_0}{\inf}\mathcal{F}(x)+\mathcal{G}(\Lambda x)=\underset{y\in X_1^\ast}{\sup}-\pmb{\F}(\Lambda^\ast y)-\pmb{\G}( -y).
 \end{equation*}
\end{theorem}
To prove Theorem \ref{pb_in_duality} using Fenchel-Rockafellar theorem, we rewrite problems \eqref{problemE} and \eqref{dual_prob} in a way similar to the problems in the statement of the Fenchel-Rockafellar theorem.

We consider the linear and bounded operator $\Lambda:X_0\to X_1$ defined by: $\Lambda(\varphi,\lambda):=(\partial_t\varphi+b\partial_s\varphi +\tilde{\lambda}, \Delta \varphi)$,
where $\partial_t\varphi + b\partial_s\varphi:=(\partial_t\varphi_i+ b_i\partial_s\varphi_i)_{i\in I}$, $\Delta \varphi :=(\Delta \varphi_{i,j})_{(i,j)\in \tilde{I}}$ with $\Delta \varphi_{i,j}=\varphi_j- \varphi_i$ and, for any $(s,i)\in [0,1]\times I,\,\tilde{\lambda}_i(\cdot,s):=\lambda_i(\cdot)$. The linear operator $\Lambda^\ast:X_1^\ast\to X_0^\ast $ is the adjoint operator of $\Lambda$.
The  functional $\F$ is defined, for any $(\varphi,\lambda)\in X_0$, by
\begin{equation*}
\F(\varphi,\lambda):= 
\left\{
\begin{array}{ll}
\sum_{i \in I}\int_0^1-\varphi_i(0,s)m_i^0(ds)+ \int_0^T D_i(t)\lambda_i(t) dt& \mbox{if }\varphi_i(T,\cdot)\leq g_i\mbox{ and }\lambda_i\geq 0\quad\forall i\in I, \\
 +\infty& \mbox{otherwise.}
\end{array}
\right.
\end{equation*}
Using that:
\begin{equation*}
\begin{array}{l}
 \langle (m,E),\Lambda(\varphi,\lambda)\rangle_{X_1^\ast, X_1} \vspace{0.2cm}
  \\
   = \sum_{i\in I} \int_0^1 \int_0^T
  (\partial_t\varphi_i(t,s) + b_i(s)\partial_s\varphi_i(t,s))m_i(ds,t) + 
  \sum_{j\in I, j\neq i}(\varphi_j(t,s)-\varphi_i(t,s))E_{i,j}(t,ds)dt\\
  +\sum_{i\in I}\int_0^T\int_0^1m_i(t,ds)\tilde{\lambda}_i(t,s)dt,
\end{array}
\end{equation*}
defining $\pmb{\F}$ as the Fenchel conjugate of $\F$, we have:
\begin{equation*}
\pmb{\F}\left(\Lambda^\ast(m,E)\right):= 
\left\{
\begin{array}{ll}
\int_0^1 \sum_{i\in I}g_i(s)m_i(T,ds) & \mbox{if }(m,E) \mbox{ is a weak solution of \eqref{evoX}}\\
&\mbox{and }\int_0^1m_i(t,ds)\leq D_i(t)\quad\forall(t,i)\in[0,T]\times I,\vspace{0.2cm}\\
+\infty & \mbox{otherwise}.
\end{array}
\right.
\end{equation*}
For any  $(x,y)\in X_1$, the functional $\G$ is defined by:
\begin{equation*}
\G(x,y):=\left\{
\begin{array}{ll}
 0 &  \mbox{if } -c_i(t,s) - x_i(t,s) +\sum_{j\in I,j\neq i}H(-y_{i,j}(t,s))\leq 0 \quad\forall (t,s,i)\in(0,T)\times(0,1)\times I,
 \\
+\infty &  \mbox{otherwise}.
\end{array}
\right.
\end{equation*}
Then, for any $(\varphi,\lambda)\in X_0$, it holds:
\begin{equation*}
\G(\Lambda(\varphi,\lambda)):=\left\{
\begin{array}{ll}
 0 &  \mbox{if } -c_i(t,s)-\partial_t\varphi_i(t,s) - b_i(t,s)\partial_s\varphi_i(t,s) -\tilde{\lambda_i}(t,s) + \sum_{j\in I,j\neq i}H(-\Delta\varphi_{i,j}(t,s))\leq 0
 \\
 & \forall (t,s,i)\in(0,T)\times(0,1)\times I,\vspace{0.2cm}\\
+\infty &  \mbox{otherwise}.
\end{array}
\right.
\end{equation*}
Using that $L$ is the Fenchel conjugate of $H$, i.e. $L(x)=\pmb{H}(x)$, one can show, as in \cite{benamou2000computational} for the quadratic case, that, for any $(v,w)\in \mathbb{R}^2$,
\begin{equation*}
\sup_{a,b \in \mathbb{R}}\,\{av+bw;\,a+H(b)\leq 0\}=\left\{
\begin{array}{ll}L(\frac{w}{v})v
 &  \mbox{if }v >0,\\
 0& \mbox{if }v =0\mbox{ and }w=0,\\
 +\infty & \mbox{otherwise}.
\end{array}
\right.
\end{equation*}
Then, with similar computations as in \cite[Lemma 4.3]{cardaliaguet2015mean}, for any $(m,E)\in X_1^\ast$, the Fenchel conjugate $\pmb{\G}$ of $\G$ satisfies:
\begin{equation}
\label{Gast0}
\begin{array}{l}
 {\pmb{\G}}(-(m,E))\\
 =
 \sup_{(x,y)\in X_1}\,\sum_{i\in I}\int_0^T\int_0^1-x_i(t,s)m_i(t,ds)dt-\sum_{j\neq i}y_{i,j}(t,s)E_{i,j}(t,ds)dt-\mathcal{G}(x,y)\vspace{0.1cm}\\
 =\sup_{(x,y)\in X_1}\,\sum_{i\in I}\int_0^T\int_0^1(-x_i(t,s)-c_i(t,s)+c_i(t,s))m_i(t,ds)-\sum_{j\neq i}y_{i,j}(t,s)E_{i,j}(t,ds)dt-\mathcal{G}(x,y)\vspace{0.2cm}\\
  =\sum_{i\in I}\int_0^T\int_0^1c_i(t,s)m_i(t,ds)dt+\sup_{(x,y)\in X_1}\,\sum_{i\in I}\int_0^T\int_0^1x_i(t,s)m_i(t,ds)+\sum_{j\neq i}y_{i,j}(t,s)E_{i,j}(t,ds)dt-\mathcal{G}(-x-c,-y)\vspace{0.2cm}\\
 =
 \left\{
 \begin{array}{ll}
\int_0^T\int_0^1\sum_{i\in I} c_i(t,s)  m_i(t,ds) + \sum_{j\neq i} L\bigg(\frac{\mbox{d}E_{i,j}}{ \mbox{d} m_i\otimes dt}(t,s)\bigg)m_i(t, ds)dt   &  \mbox{if }m>0,E\geq 0 \mbox{ and }E\ll m,\vspace{0.1cm}\\
 0 & \mbox{if }m=0 \mbox{ and }E=0,\vspace{0.1cm}\\
 +\infty &\mbox{otherwise}.
 \end{array}
 \right.
\end{array}
\end{equation}
By the definition of $\mathcal{F}$ and $\mathcal{G}$ above, one has
\begin{equation}\label{dual_prob_eq_f_g}
\inf_{(\varphi,\lambda)\in \mathcal{K}_0}\,A(\varphi,\lambda)=
\inf_{(\varphi,\lambda)\in X_0}\F(\varphi,\lambda)+\G(\Lambda(\varphi,\lambda)),
\end{equation}
and
\begin{equation}\label{primal_prob_eq_f_g}
\min_{(m,E)\in \mathcal{S}(m^0,D)}\,\tilde{J}(m,E)=
\min_{(m,E)\in X_1^\ast}\,
\pmb{\F}(\Lambda^\ast(m,E))+\pmb{\G}(-(m,E)).
\end{equation}
Before to apply the Fenchel-Rockafellar theorem to our particular case, we need to prove that Problem \eqref{dual_prob} is finite (Lemma \ref{dual_finite})  and that there exists $(\varphi,\lambda)$ in $X_0$ such that $\mathcal{F}(\varphi,\lambda)<+\infty$ and $\mathcal{G}$ is continuous at $\Lambda(\varphi,\lambda)$ (Lemma \ref{contraint_qualifiation}).

\begin{lemma}
\label{dual_finite}
$\inf_{(\varphi,\lambda)\in \mathcal{K}_0}\,A(\varphi,\lambda)$ is finite.
\end{lemma}

\begin{proof}
Let $(\varphi,\lambda)\in \mathcal{K}_0$. Since $\varphi$ is in $C^1([0,T]\times [0,1],\mathbb{R}^{\vert I \vert})$ and satisfies \eqref{hjb_ineq}, using that $H$ is non-negative, we have, for any $(i,s)\in I\times[0,1]$, $\varphi_i(0,s)\leq T \Vert c\Vert_\infty+ \int_0^T\lambda_i(\tau)d\tau + \Vert g\Vert_\infty$. Setting $Q:=-\vert I \vert( T \Vert c\Vert_\infty +\Vert g\Vert_\infty)$, one has:
\begin{equation*}
Q+\sum_{i\in I}
 \int_0^T\lambda_i(t)\bigg(D_i(t)-\int_0^1m^0_i(ds)\bigg)dt
 \leq A(\varphi,\lambda).
\end{equation*}
Since $\lambda\geq0$, we deduce from the Assumption \ref{hyp_on_D} and previous inequality that $Q\leq \inf_{(\varphi,\lambda)\in \mathcal{K}_0}\,A(\varphi,\lambda)$.
\end{proof}

\begin{lemma}
\label{contraint_qualifiation}
There exists $(\varphi,\lambda)\in X_0$ such that $\mathcal{F}(\varphi,\lambda)< \infty$ and $\mathcal{G}$ is continuous at $\Lambda(\varphi,\lambda)$.
\end{lemma}
\begin{proof}
Let $\varphi$ and $\lambda$ be such that, for any $(t,s,i)\in [0,T]\times [0,1]\times I$,
\begin{equation*}
\varphi_i(t,s) = -\max_{i\in I}( \Vert g_i \Vert_\infty) - 1,
\end{equation*}
and
\begin{equation*}
\lambda_i(t) := \Vert c_i\Vert_\infty + 1,
\end{equation*}
Functions $\varphi$ and $\lambda$ being constant, it holds that $(\varphi,\lambda)\in X_0$ and  $\mathcal{F}(\varphi,\lambda)<\infty$. Also,
from the choice of $\varphi$ and $\lambda$, it  follows that, for any  $i\in I$, $s\in[0,1]$ and $t\in [0,T]$,
\begin{equation*}
 -c_i(t,s)-\partial_t\varphi_i(t,s) - b_i(t,s)\partial_s\varphi_i(t,s) -\lambda_i(t,s) + \sum_{j\in I,j\neq i}H(\varphi_{i}(t,s)-\varphi_{j}(t,s))< 0.
\end{equation*}
Thus, $\mathcal{G}$ is continuous at $\Lambda(\varphi,\lambda)$.  
\end{proof}

We are now ready to prove Theorem \ref{pb_in_duality}.
\begin{proof}
Using Lemmas \ref{contraint_qualifiation} and \ref{dual_finite} and observing that \eqref{dual_prob_eq_f_g} and \eqref{primal_prob_eq_f_g}, the conclusion follows by applying the Fenchel-Rockafellar duality theorem.
\end{proof}

Problem \eqref{dual_prob} may not have an optimum. The next proposition states that if the space of solutions $X_0$ is relaxed, then one can find a candidate $(\varphi,\lambda)$ in $ BV((0,T)\times (0,1),\mathbb{R}^{\vert I \vert})\times \mathcal{M}^+([0,T],\mathbb{R}_+^{\vert I \vert})$, such that $\tilde{A}(\varphi,\lambda)$, where $\tilde{A}$ is defined in \eqref{def_tilde_A_0}, is equal to the value of Problem \eqref{dual_prob}.
\begin{proposition}\label{eq_inf_relax}
There exists $\lambda\in \mathcal{M}^+([0,T],\mathbb{R}_+^{\vert I \vert})$ such that:
\begin{equation*}
  \tilde{A}(\varphi^\lambda,\lambda)=  \inf_{(\phi,\mu)\in 
\mathcal{K}_0}{A}(\phi,\mu),
\end{equation*}
where $\varphi^\lambda$ is defined in \eqref{def_varphi}.
\end{proposition}
Before proving this result, we need the following lemma.
\begin{lemma}
\label{borne_lambda}
For any $K>0$,  there exists a constant $ C>0$ such that, for any $(\varphi,\lambda)\in \mathcal{K}_0$ satisfying $A(\varphi,\lambda)\leq K$, we have:
\begin{equation*}
\sum_{i\in I}\int_0^T\lambda_i(t)dt \leq C.
\end{equation*}
\end{lemma}
\begin{proof}
Let $K\in \mathbb{R}_+^\ast$ and 
$(\varphi,\lambda)\in \mathcal{K}_0$ be such that ${A}(\varphi,\lambda)\leq K$. Since  $(\varphi,\lambda)\in \mathcal{K}_0$, one can show, as in the proof of Lemma \ref{dual_finite}, that, for any $(i,s)\in I\times [0,1]$,
\begin{equation*}
    \varphi_i(0,s)\leq T\Vert c \Vert_\infty + \Vert g \Vert_\infty + \sum_{i\in I}\int_0^T\lambda_i(t)dt.
\end{equation*}
Therefore, recalling that ${A}(\varphi,\lambda)\leq K$:
\begin{equation*}
    \sum_{i\in I}\int_0^T\lambda_i(t)(D_i(t)-m_i([0,1]))dt\leq K+ T\Vert c \Vert_\infty + \Vert g \Vert_\infty .
\end{equation*}
From Assumption \ref{hyp_on_D}, there exists $\varepsilon^0>0$ such that $D_i(t)-m_i([0,1])>\varepsilon^0$ for any $(t,i)\in [0,T]\times I$. Setting $C:=( K+ T\Vert c \Vert_\infty + \Vert g \Vert_\infty)/\varepsilon^0$, we have:
\begin{equation*}
     \sum_{i\in I}\int_0^T\lambda_i(t)dt\leq C.
\end{equation*}
\end{proof}
We are now ready to prove Proposition \ref{eq_inf_relax}.

\begin{proof}[Proof of Proposition \ref{eq_inf_relax}]
Let $\{(\phi^n,\lambda^n)\}_n$ be a minimizing sequence of \eqref{dual_prob}. For any $n\in \mathbb{N}$, we consider $\varphi^n$ the classical solution of \eqref{hjb_0} associated with $\lambda^n$. 
Using that $\varphi^n$ is a classical solution and that $\phi^n$ satisfies \eqref{hjb_ineq} on $[0,T]\times [0,1]\times I$, and defining $\gamma$ on $[0,T]$ by
\begin{equation*}
    \gamma(t):=\underset{(x,i)\in [0,1]\times I}{\max}(\phi^n_i-\varphi_i^n)(t,x)),
\end{equation*}
one can show by adapting the proof Lemma \ref{a_priori_estimate} that $\gamma<0$ on $[0,T]$ and thus, that the following comparison holds: $\phi^n\leq \varphi^n$.
Thus, $A(\varphi^n,\lambda^n)\leq A(\phi^n,\lambda^n)$ for any $n\in \mathbb{N}$. Therefore, $\{(\varphi^n,\lambda^n)\}_n$ is also a minimizing sequence and, there exist $K>0$ and $n^0>0$ such that for any $n\geq n^0$ one has $A(\varphi^n,\lambda^n)\leq K$.
From Lemma \ref{borne_lambda}, the sequence $\{\sum_{i\in I}\int_0^T\lambda_i^n(t)dt\}_n$ is uniformly bounded.
Thus, a subsequence of $\{\lambda^n\}_n$ weakly$^\ast$ converges to a measure $\lambda\in \mathcal{M}^+([0,T],\mathbb{R}_+^{\vert I \vert})$ w.r.t. the weak$^\ast$ topology in $ \mathcal{M}^+([0,T],\mathbb{R}_+^{\vert I \vert})$ \cite[Theorem 1.59]{ambrosio2000functions}. We set, for any $n\in \mathbb{N}$,
\begin{equation*}
    \varphi^\lambda:=\Theta(\lambda)+L^\lambda\quad\mbox{ and }\quad\varphi^n:=\Theta(\lambda^n)+L^{\lambda^n},
\end{equation*}
where $\Theta$ is defined in \eqref{def_theta}. By the same arguments as in the proof of Lemma \ref{approx_lambda_smooth}, one has $ \lim_{n\to \infty}\Vert \varphi^\lambda(0,\cdot)-\varphi^n(0,\cdot)\Vert_\infty=0$ up to a subsequence of $\{\varphi^n\}_n$ and, therefore, we have:
\begin{equation*}
    \lim_{n\to \infty}A(\varphi^n, \lambda^n)=\tilde{A}(\varphi^\lambda,\lambda),
\end{equation*}
and the conclusion follows.
\end{proof}

\section{Characterization of the minimizers}
\label{charac_mini}

The purpose of this section is to define and characterize the solutions of Problem \eqref{problemE}. We show that the following system gives optimality conditions for \eqref{problemE}:
\begin{equation}
\label{forward_backward_system}
\left\{
\begin{array}{ll}
-\partial_t\varphi_i-b_i\partial_s\varphi_i-c_i-\lambda_i+\sum_{j\in I,j\neq i}H(\varphi_i-\varphi_j)= 0 
&\mbox{on }(0,T)\times (0,1)\times I,
\\
\partial_t m_i+\partial_s(m_i b_i)
+\sum_{j\neq i}H'(\varphi_i-\varphi_j)m_i -H'(\varphi_j-\varphi_i)m_{j} =0 
&\mbox{on }(0,T)\times (0,1)\times I,\vspace{0.2cm}\\
m_i(0,s) =  m_i^0(s),\,  \varphi_i(T,s) =  g_i(s)+\lambda_i(\{T\})
&\mbox{on } (0,1)\times I,
\\
\int_0^1m_i(t,ds)-D_i(t)\leq 0,\,\lambda\geq 0
&\mbox{on }[0,T]\times I,\\
\sum_{i\in I}\int_0^T\left(\int_0^1 m_i(t,ds)-D_i(t)\right)\lambda_i(dt)=0.
\end{array}
\right.
\end{equation}

The notion of weak solutions of system \eqref{forward_backward_system} is given in the following definition.
\begin{definition}
\label{weak_sol_syst}
A triplet $(\varphi, \lambda, m)\in (\Lip([0,T]\times [0,1],\mathbb{R}^{\vert I \vert})+BV([0,T],\mathbb{R}^{\vert I \vert}))\times \mathcal{M}^+([0,T],\mathbb{R}_+^{\vert I \vert})\times C^0([0,T],\mathcal{P}([0,1]\times I))$ is called a weak solution of \eqref{forward_backward_system} if it satisfies the following conditions:
\begin{enumerate}
\item \label{weak_sol_hjb_syst} The function $\varphi$ is a weak solution of \eqref{hjb_0}, associated with $\lambda$ in the sense of Definition \ref{def_weak_sol_hjb};
 \item $m$ satisfies the continuity equation:
\begin{equation*}
\partial_t m_i+\partial_s(m_i b_i)
+\sum_{j\neq i}H'(\varphi_i-\varphi_j)m_i -H'(\varphi_j-\varphi_i)m_{j} =0,\quad m_i(0,\cdot)=m_i^0,
\end{equation*}
in the sense of Definition \ref{weak_sol_cont_equ}, with $\alpha_{i,j}:=H'(\varphi_i-\varphi_j)$;
\item it holds that, for any $t\in [0,T]$, 
\begin{equation*}
\begin{array}{lll}
\int_0^1m_i(t,ds)-D_i(t)\leq 0 & \mbox{ and } &\sum_{i\in I}\int_0^T\left(\int_0^1 m_i(t,ds)-D_i(t)\right)\lambda_i(dt)=0.
\end{array}
\end{equation*}
\end{enumerate}
\end{definition}
\begin{remark}
Since $\varphi\in \Lip([0,T]\times [0,1],\mathbb{R}^{\vert I \vert})+BV([0,T],\mathbb{R}^{\vert I \vert})$ and $H'$ is Lipschitz continuous, the control $\alpha_{i,j}:=H'(\varphi_i-\varphi_j)$ is bounded on $[0,T]\times [0,1]\times I$ and $\partial_s\alpha_{i,j}\in L^\infty([0,T]\times [0,1])$. Thus, $\alpha_{i,j}$ is in $L_{m_i\otimes dt}^2([0,T]\times[0,1])$ and the forward equation in \eqref{forward_backward_system} makes sense.
\end{remark}

The following theorem states the optimality conditions of Problem \eqref{problemE}. We recall that the definition of $\mathcal{S}(m^0,D)$ is given in \eqref{def_set_CE}.
\begin{theorem}
\label{optimality_conditions}
\begin{enumerate}
\item \label{minimizer_implies_weak_solution}If $(m,E)\in \mathcal{S}(m^0,D)$ is a minimizer of Problem \eqref{problemE}, and $\varphi$ a weak solution of \eqref{hjb_0} in the sense of Definition \ref{def_weak_sol_hjb} associated with $\lambda$ satisfying $\tilde{A}(\varphi,\lambda)=\inf_{(\phi,\mu)\in 
\mathcal{K}_0}{A}(\phi,\mu)$, then $(\varphi, \lambda, m)$ is a weak solution of \eqref{forward_backward_system} and $\frac{\mbox{d}E_{i,j}}{\mbox{d}m_i\otimes dt}=H'(\varphi_i-\varphi_j)$ on $\{m_i>0\}$ for any $i,j\in I$.
\item \label{weak_sol_implies_minimizer}Conversely, if $(\varphi,\lambda,m)$ is a weak solution of \eqref{forward_backward_system}, then $\tilde{A}(\varphi,\lambda)=\inf_{(\phi,\mu)\in 
\mathcal{K}_0}{A}(\phi,\mu)$ and there exists $E$, defined for any $i,j\in I$ by $\frac{\mbox{d}E_{i,j}}{\mbox{d}m_i\otimes dt}:=H'(\varphi_i-\varphi_j)$, such that $(m,E)\in \mathcal{S}(m^0,D)$ is a minimizer of  \eqref{problemE}.
\end{enumerate}
\end{theorem}
\begin{remark}
If $(\varphi,\lambda,m)$ is a weak solution of \eqref{forward_backward_system}, then $(\varphi,\lambda)$ is a minimizer of a relaxed version of Problem \eqref{dual_prob}, i.e.  $(\varphi,\lambda)$ is the minimum of $\tilde{A}$, defined in \eqref{def_tilde_A_0}, over the space $ (\Lip([0,T]\times [0,1],\mathbb{R}^{\vert I \vert})+BV([0,T],\mathbb{R}^{\vert I \vert}))\times \mathcal{M}^+([0,T],\mathbb{R}_+^{\vert I \vert})$.
\end{remark}

\begin{remark}\label{final_time_atom}
    The term $\lambda_i(\{T\})$, in the final condition of the backward equation of the system \eqref{forward_backward_system}, can be compared with the Lagrange multiplier $\beta_T$ in \cite[Equation 1.2]{cardaliaguet2016first} of the density constraint at the final time. The term $\lambda_i(\{T\})$ ensures that the constraint \eqref{congestion_ineq_D} is satisfied at time $t=T$ and is strictly positive only if the constraint \eqref{congestion_ineq_D} is saturated. 
\end{remark}

\begin{remark}\label{uniqueness_disc}
One can not expect uniqueness of the solution of Problem \eqref{opt:J} and of the system \eqref{forward_backward_system}. This is mainly due to the absence of a strictly convex function w.r.t. the distribution $m$ in the definition of the objective function $J$ in \eqref{obj_formulation}.
\end{remark}

\subsection{Proof of Theorem 6.1}
Before the proof of Theorem \ref{forward_backward_system}, we make the following remark. We recall that the definition of $\tilde{A}$ is given in \eqref{def_tilde_A_0}, of $\tilde{J}$ in \eqref{def_B} and of $A$ in \eqref{def_dual_fun}.

{\color{red}}
\begin{remark}
\label{ineg_phi_0_phi_T}
For any $\lambda\in \mathcal{M}^+([0,T],\mathbb{R}_+^{\vert I \vert})$ and any $(m,E)\in \mathcal{S}(m^0,D)$ one has $-\tilde{J}(m,E)\leq \tilde{A}(\varphi^\lambda,\lambda)$. Indeed, considering a sequence $\{(\lambda^n,\varphi^n)\}_n$ defined as in Lemma \ref{approx_lambda_smooth} and using the proof of Theorem \ref{pb_in_duality}, we get: $-\tilde{J}(m,E)\leq A(\varphi^n,\lambda^n)$ and, therefore, $-\tilde{J}(m,E)\leq \tilde{A}(\varphi^\lambda,\lambda)$. 
\end{remark}

\begin{proof}[Proof of Theorem \ref{optimality_conditions}]\ref{minimizer_implies_weak_solution}. 
By Theorem \ref{pb_in_duality}, one has:
\begin{equation*}
\inf_{(\hat{\varphi},\hat{\lambda})\in \mathcal{K}_0}\,{A}(\hat{\varphi},\hat{\lambda})=
- \underset{(\hat{m},\hat{E})\in \mathcal{S}(m^0,D)}{\inf}\,\tilde{J}(\hat{m},\hat{E}),
\end{equation*}
and thus
\begin{equation}
\label{eg_opt_obj_functions}
\sum_{i\in I} \int_0^1g_im_i(T)-\varphi_i(0)m_i^0 +\int_0^TD_i\lambda_i+
\int_0^T\int_0^1\Bigg(c_i+\sum_{j\neq i}L\bigg(\frac{\mbox{d}E_{i,j}}{\mbox{d}m_i\otimes dt}\bigg)\Bigg)m_i = 0.
\end{equation}
We want to show that $E_{i,j}=H'(\varphi_i-\varphi_j)m_i$.
We consider a sequence $\{(\lambda^n,\varphi^n)\}_n$  defined as in Lemma \ref{approx_lambda_smooth}. 
For any $n\in \mathbb{N}$, $\varphi^n$ is smooth enough to be a test function for the weak formulation of \eqref{fk} satisfied by $m$. 
According to  Lemma \ref{approx_lambda_smooth} and the fact that $D\in C^0([0,T],\mathbb{R}^{\vert I \vert})$, it holds that, for any $n\in \mathbb{N}$ and $i\in I$,
\begin{equation*}
\begin{array}{l}
\sum_{i\in I} \int_0^1g_im_i(T)-\varphi_i(0)m_i^0  + \int_0^T\lambda_iD_i 
\\= \lim_{n\to \infty}\sum_{i\in I} \int_0^1g_im_i(T)-\varphi_i^n(0)m_i^0  + \int_0^T D_i\lambda^n_i\\
 =  \lim_{n\to \infty}\sum_{i\in I} \int_0^T\int_0^1\Bigg(\partial_t\varphi^n_i+b_i\partial_s(\varphi^n_i)
 +\sum_{j\in I,j\neq i}(\varphi_j^n-\varphi_i^n)\frac{\mbox{d}E_{i,j}}{\mbox{d}m_i\otimes dt}
 \Bigg)m_i  + \int_0^T\lambda_i^n D_i\vspace{0.1cm}
 \\
  =  \lim_{n\to \infty}\sum_{i\in I} \int_0^T\int_0^1\Bigg(-c_i+\sum_{j\in I,j\neq i}H(\varphi_i^n-\varphi_j^n) +\sum_{j\in I,j\neq i}(\varphi_j^n-\varphi_i^n)\frac{\mbox{d}E_{i,j}}{\mbox{d}m_i\otimes dt}
 \Bigg)m_i  + \int_0^T\lambda_i^n\left(D_i-\int_0^1m_i\right).
\end{array}
\end{equation*}
By the previous equality and \eqref{eg_opt_obj_functions},
\begin{equation*}
\lim_{n\to \infty}\sum_{i\in I} \int_0^T\int_0^1\Bigg(\sum_{j\in I,j\neq i}H(\varphi_i^n-\varphi_j^n) +L\bigg(\frac{\mbox{d}E_{i,j}}{\mbox{d}m_i\otimes dt}\bigg)
+(\varphi_j^n-\varphi_i^n)\frac{\mbox{d}E_{i,j}}{\mbox{d}m_i\otimes dt}
 \Bigg)m_i  + \int_0^T\lambda_i^n\left(D_i-\int_0^1m_i\right)  = 0.
\end{equation*}
According to Lemma
\ref{approx_lambda_smooth}, for a.e. $t\in [0,T]$, the sequence $\{\varphi(t,\cdot)\}_{n}$ converges uniformly to $\varphi$.
By the continuity of $H$ and the dominated convergence theorem, we get $\textstyle \int_0^T\int_0^1m_iH(\varphi_i^n-\varphi_j^n)$ converges to $\textstyle \int_0^T \int_0^1m_iH(\varphi_i-\varphi_j)$ for any $i,j\in I$.
Since $(m,E)$ is a solution of \eqref{problemE}, $\tilde{J}(m,E)$ is finite. Then 
one can show that,
for any $i,j\in I$, $\textstyle \int_0^T\int_0^1 E_{i,j}<\infty$. Applying  dominated convergence theorem, we can show that: $\textstyle\int_0^T\int_0^1(\varphi_j^n-\varphi_i^n)E_{i,j}$ converges to $\textstyle \int_0^T \int_0^1(\varphi_j-\varphi_i)E_{i,j}$.
Since, for any $i\in I$, the map $t\mapsto D_i(t)-\int_0^1m_i(t,ds)$ is continuous,
the weak$^\ast$ convergence of ${\lambda^n}_n$ to $\lambda$ in $\mathcal{M}^+([0,T],\mathbb{R}_+^{\vert I \vert})$ gives:
$$\lim_{n\to \infty}\sum_{i\in I}  \int_0^T\lambda_i^n\left(D_i-\int_0^1m_i\right)=\sum_{i\in I}  \int_0^T\lambda_i\left(D_i-\int_0^1m_i\right).$$
Thus,
\begin{equation}
\label{lim_eg_opt_obj_functions}
\int_0^T\int_0^1\Bigg(\sum_{j\in I,j\neq i}H(\varphi_i-\varphi_j) +L\bigg(\frac{\mbox{d}E_{i,j}}{\mbox{d}m_i\otimes dt}\bigg)
+(\varphi_j-\varphi_i)\frac{\mbox{d}E_{i,j}}{\mbox{d}m_i\otimes dt}
 \Bigg)m_i  + \int_0^T\lambda_i\left(D_i-\int_0^1m_i\right)  = 0.
\end{equation}
Since $\lambda\geq 0$ and $\int_0^1m_i(t,ds)\leq D_i(t)$ for any $t\in [0,T]$, one has, for any $i\in I$ and $t\in [0,T]$,
\begin{equation}
\label{lim_inf_lamb_ineq_m_n}
0\leq \int_0^T\left(D_i(t)-\int_0^1m_i(t,ds)\right)\lambda_i(dt).
\end{equation}
Recalling that  $\pmb{L}(p) = H(p)$, we have $L(p) + H(q)-pq\geq 0$ for any $p,q\in \mathbb{R}$. Thus, by inequality \eqref{lim_inf_lamb_ineq_m_n} and equality \eqref{lim_eg_opt_obj_functions}, one deduces 
\begin{equation*}
    \sum_{j\in I,j\neq i}H(\varphi_i-\varphi_j) +L\bigg(\frac{\mbox{d}E_{i,j}}{\mbox{d}m_i\otimes dt}\bigg)
+(\varphi_j-\varphi_i)\frac{\mbox{d}E_{i,j}}{\mbox{d}m_i\otimes dt} =0 \quad m-\mbox{a.e..}
\end{equation*}
Therefore,
\begin{equation}
\label{id_E}
\frac{\mbox{d}E_{i,j}}{\mbox{d}m_i\otimes dt}(t,s)=H'(\varphi_i(t,s)-\varphi_j(t,s))\quad
m-\mbox{a.e. },
\end{equation} 
and inequality \eqref{lim_inf_lamb_ineq_m_n} becomes an equality. Thus, 
\begin{equation*}
    \int_0^T\left(D_i-\int_0^1m_i(t,ds)\right)\lambda_i(dt)=0.
\end{equation*}
By equality \eqref{id_E}, the properties of $H$ and the fact that $\varphi\in \Lip([0,T]\times [0,1],\mathbb{R}^{\vert I \vert})+BV([0,T],\mathbb{R}^{\vert I \vert})$, one has $\frac{\mbox{d}E_{i,j}}{\mbox{d}m_i\otimes dt}\in L^\infty([0,T]\times [0,1])$ and $\partial_s \Big(\frac{\mbox{d}E_{i,j}}{\mbox{d}m_i\otimes dt}\Big)\in L^\infty((0,T)\times (0,1))$. Thus, by Proposition \ref{weak_sol_fk_coro} in Appendix \ref{appendices_section}, we deduce that 
$m \in \Lip([0,T],\mathcal{P}( [0,1]\times I))$.

\ref{weak_sol_implies_minimizer}. We assume now that $(\varphi,\lambda,m)$ is a weak solution of \eqref{forward_backward_system}. Since $\varphi$ is in $\Lip([0,T]\times [0,1],\mathbb{R}^{\vert I \vert})+BV([0,T],\mathbb{R}^{\vert I \vert})$ and $\lambda$ is a finite measure, the quantity $\tilde{A}(\varphi,\lambda)$ is well defined. We want to show that $\tilde{A}(\varphi,\lambda)+\tilde{J}(m,E)=0$. 
We approximate $(\lambda,\varphi)$ by the sequence $\{(\lambda^n,\varphi^n)\}_n$ defined in Lemma \ref{approx_lambda_smooth}.(\ref{minimizer_implies_weak_solution}). For any $n$, $\varphi^n$ is smooth enough to be considered as a test function for the equation \eqref{fk} satisfied in the weak sense by $m$. We have, for any $i\in I$,
\begin{equation}
\label{test_phi_weak_sol_m}
 \sum_{i\in I} \int_0^1g_im_i(T)-\varphi^n_i(0)m_i^0  +\sum_{i\in I}\int_0^T\int_0^1-m_ib_i\partial_s\varphi_i^n -m_i\partial_t\varphi_i^n+\sum_{j\in I,j\neq i}(\varphi_i^n-\varphi^n_j)H'(\varphi_i-\varphi_j)m_i= 0.
\end{equation}
For any $i\in I$, $\varphi^n_i$ is a classical solution of \eqref{hjb_0} associated with $\lambda^n$. Multiplying \eqref{hjb_0} by $m_i$, summing over $I$ and integrating over $[0,T]\times [0,1]$, we have:
\begin{equation}
\label{varphi_classic_sol_multiplied_m}
\sum_{i\in I}\int_0^T\int_0^1-m_i\partial_t\varphi_i^n-
m_ib_i\partial_s  \varphi_i^n -m_ic_i-m_i\lambda^n +\sum_{j\in I,j\neq i}H(\varphi_i^n-\varphi_j^n)m_i =0.
\end{equation}
Combining \eqref{test_phi_weak_sol_m} and \eqref{varphi_classic_sol_multiplied_m} yields
\begin{equation*}
\sum_{i\in I} \int_0^1g_im_i(T)-\varphi^n_i(0)m_i^0  +\sum_{i\in I}\int_0^T\int_0^1  c_i m_i +\lambda_i^n m_i  +m_i\bigg(\sum_{j\in I,j\neq i}H'(\varphi_i-\varphi_j)(\varphi^n_i-\varphi^n_j) - H(\varphi_i^n-\varphi_j^n)\bigg)= 0.
\end{equation*}
Since $(\varphi,\lambda,m)$ is a weak solution of \eqref{forward_backward_system}, by Lemma \ref{approx_lambda_smooth}, and letting $n$ tend to infinity, one deduces:
\begin{equation*}
\sum_{i\in I} \int_0^1g_im_i(T)-\varphi_i(0)m_i^0  +\sum_{i\in I}\int_0^1 D_i \lambda_i +\sum_{i\in I} \int_0^T\int_0^1  c_i m_i  +m_i\bigg(\sum_{j\in I,j\neq i}H'(\varphi_i-\varphi_j)(\varphi_i-\varphi_j) - H(\varphi_i-\varphi_j)\bigg)= 0.
\end{equation*}
By the definition of $L$ and $H$, we have:
\begin{equation*}
\sum_{i\in I} \int_0^1g_im_i(T)-\varphi_i(0)m_i^0  +\sum_{i\in I} \int_0^1 D_i \lambda_i +\sum_{i\in I} \int_0^T\int_0^1  c_i m_i  +m_i\bigg(\sum_{j\in I,j\neq i} L(H'(\varphi_i-\varphi_j))\bigg)= 0.
\end{equation*}
By the definition of $\tilde{A}$ in \eqref{def_tilde_A_0} and $\tilde{J}$ in \eqref{def_B}, we have $\tilde{A}(\varphi,\lambda)+\tilde{J}(m,E)=0$. Finally, by Remark \ref{ineg_phi_0_phi_T}, one deduces that $(m,E)$ is a minimizer of \eqref{problemE}.
\end{proof}

\subsection{Proof of Theorem 2.1}\label{proof_main_th}
We are now ready to prove our main theorem by using Theorem \ref{optimality_conditions} and applying the change of variable $\alpha_{i,j}:=\frac{\mathrm{d}E_{i,j}}{m_i}$. 
\begin{proof}[Proof of Theorem \ref{main_results}] The existence of a solution to Problem \eqref{opt:J} is given by Lemma \ref{Existence_sol_prob_init}.

\ref{min_imp_wea} This statement is proved by  Theorem \ref{optimality_conditions}.\ref{minimizer_implies_weak_solution}

\ref{wea_imp_min} This point is given by  Theorem \ref{optimality_conditions}.\ref{weak_sol_implies_minimizer}.

\ref{min_imp_reg} The uniform bound on $\alpha$ and $\partial_s\alpha$ are deduced by Theorem \ref{main_results}.\ref{min_imp_wea}, using the fact that $H$ has a globally Lipschitz continuous gradient and that $\varphi$ is in $\Lip([0,T]\times [0,1],\mathbb{R}^{\vert I \vert})+BV([0,T],\mathbb{R}^{\vert I \vert})$. The time regularity of $m$ is obtained by Proposition \ref{regularity_weak_solution} in Appendix \ref{appendices_section}.
\end{proof}
We now prove Proposition \ref{reg_m_sol}.
\begin{proof}[Proof of Proposition \ref{reg_m_sol}]
Let $(m,\alpha)$ be a solution of Problem \eqref{opt:J}
 and $\mu^0$ be the density of $m^0$ w.r.t. the Lebesgue measure. By a fixed point argument  \cite[Theorem 5.7]{brezis2011functional}, it is easy to check that there exists a unique solution $\mu\in\Lip([0,T]\times [0,1],\mathbb{R}^{\vert I \vert})$ of the following equation on $[0,T]\times [0,1]\times I$:
\begin{equation}
\label{int_formulation_G}
\begin{array}{ll}
  \mu_i(t,s)=   & \mu_i^0(S_i^{t,s}(0))+\int_0^t \mu_i(\tau,S_i^{t,s}(\tau))\partial_sb_i(S_i^{t,s}(\tau))  d\tau\\
     & +\int_0^t\sum_{j\in I, j\neq i}-\alpha_{i,j}(\tau,S_i^{t,s}(\tau))
\mu_i(\tau,S_i^{t,s}(\tau))+\alpha_{j,i}(\tau,S_i^{t,s}(\tau))\mu_j(\tau,S_i^{t,s}(\tau))d\tau.
\end{array}
\end{equation}
To prove the existence of a fixed point, it is crucial to have $\mu^0$ in $C^1([0,1],\mathbb{R}^{\vert I \vert})$ in order to define properly the quantity $\mu_i^0(S_i^{t,s}(0))$ and to obtain uniform bounds and Lipschitz estimates on $\mu_i$.
Denote by $\mathcal{L}$ the Lebesgue measure on $[0,T]\times [0,1]$,
the conclusion follows by  proving that $\mu\mathcal{L}$ is the unique weak solution of \eqref{fk}.
\end{proof}

\section{Stability with respect to the constraint and the initial distribution}\label{lip_value_data}

In this section we study how the value of Problem \eqref{opt:J} depends on the initial distribution $m^0\in \mathcal{P}([0,1]\times I)$ and on the parameter $D\in C^0\big([0,T],(\mathbb{R}_+^\ast)^{\vert I \vert}\big)$ of the constraint \eqref{congestion_ineq_D}. We endow the space  $\Omega:=\mathcal{P}( [0,1]\times I)\times C^0([0,T],(\mathbb{R}_+^\ast)^{\vert I \vert})$ with the distance $\mathcal{D}_\Omega$ defined by:
\begin{equation*}
    \mathcal{D}_\Omega((m^0,D),(\bar{m}^0,\bar{D})):=\mathcal{W}(m^0,\bar{m}^0)+\Vert D-\bar{D} \Vert_\infty,
\end{equation*}
where $\mathcal{W}$ is the $1$-Wasserstein distance on $\mathcal{P}( [0,1]\times I)$. We recall that the definition of $\Omega_\varepsilon$ is given in \eqref{def_omega_eps}. 
For any $\varepsilon>0$ we consider the function $\mathcal{V}:\Omega_\varepsilon\to \mathbb{R}$ defined by:
\begin{equation}\label{def_v}
    \mathcal{V}(m^0,D):=\inf_{(m,E)\in \mathcal{S}(m^0,D)}\,\tilde{J}(m,E),
\end{equation}
where the set $\mathcal{S}(m^0,D)$ is defined in \eqref{def_set_CE}.

The main result of this section is the following proposition, which shows the Lipschitz continuity of the value of the problem \eqref{opt:J} w.r.t. the initial distribution and the congestion constraint \eqref{congestion_ineq_D}.
\begin{proposition}\label{lip_prop_ata}
For any $\varepsilon>0$, $\mathcal{V}$ is Lipschitz continuous on $\Omega_\varepsilon$ w.r.t. the distance $\mathcal{D}_\Omega$.
\end{proposition}

To prove the proposition, we need to introduce some lemmas.
For any $(m^0,D)\in \Omega_\varepsilon $, we consider the function $A[m^0,D]:C^0([0,T]\times [0,1],\mathbb{R}^{\vert I \vert})\times C^0([0,T],\mathbb{R}_+^{\vert I \vert}) \to \mathbb{R}$ defined by:
\begin{equation*}
    A[m^0,D](\varphi,\lambda):=\sum_{i \in I}\int_0^1  - \varphi_i(0,s)m^0_i(ds) + \int_0^T\lambda_i(t)D_i(t)dt.
\end{equation*}
The following result gives some properties of the function $\mathcal{V}$.
\begin{lemma}\label{bound_A}
For any $\varepsilon>0$, the function $\mathcal{V}$ is bounded independently of $\varepsilon$, convex and l.s.c. on $\Omega_\varepsilon$.
\end{lemma}
\begin{proof}
Let  $\varepsilon>0$ and $C:=(T\Vert c \Vert_\infty + \Vert g \Vert_\infty)$. By the definition of $\mathcal{V}$ in \eqref{def_v}, one can show, for any $(m^0,D)\in \Omega_\varepsilon$, that $\vert \mathcal{V}(m^0,D) \vert \leq C$.
For any $(m^0,D)\in \Omega_\varepsilon$, the duality result in Theorem \ref{pb_in_duality} gives
\begin{equation}
\label{formulation_v_phi_lambda}
\mathcal{V}(m^0,D)=\sup_{(\varphi,\lambda)\in \mathcal{K}_0}-{A}[m^0,D](\varphi,\lambda).
\end{equation}
Since $\mathcal{V}$ is the supremum of continuous and linear functions, we deduce that $\mathcal{V}$ is convex and l.s.c. on $ \Omega_\varepsilon$. 
\end{proof}
For any $\varepsilon>0$ and $(m^0,D)\in \Omega_\varepsilon$, we know by Proposition \ref{eq_inf_relax} that there exists $(\varphi^{m^0,D},\lambda^{m^0,D})\in (\Lip([0,T]\times [0,1],\mathbb{R}^{\vert I \vert})+BV([0,T],\mathbb{R}^{\vert I \vert}))\times \mathcal{M}^+([0,T],\mathbb{R}_+^{\vert I \vert})$ such that $\varphi^{m^0,D}$ is a weak solution of \eqref{hjb_0}, in the sense of Definition \ref{def_weak_sol_hjb}, associated with $\lambda^{m^0,D}$ and such that $(\varphi^{m^0,D},\lambda^{m^0,D})$ satisfies:
\begin{equation*}
\sum_{i \in I}\int_0^1   \varphi^{m^0,D}_i(0,s)m^0_i(ds) - \int_0^TD_i(t)\lambda^{m^0,D}_i(dt)
=-\inf_{(\varphi,\lambda)\in \mathcal{K}_0}{A}[m^0,D](\varphi,\lambda)
= \mathcal{V}(m^0,D).
\end{equation*}
The next lemma provides an estimate on $\varphi^{m^0,D}$ and $\lambda^{m^0,D}$ for any $\varepsilon>0$ and $(m^0,\lambda)\in \Omega_\varepsilon$.
\begin{lemma}
\label{bound_phi_lambda_lip_const}
For any $\varepsilon>0$, there exists $C>0$ such that, for any $(m^0,D)\in \Omega_\varepsilon$,
$$ \max\big(\Vert \varphi^{m^0,D}\Vert_\infty, \,\Vert \partial_s\varphi^{m^0,D}\Vert_\infty,\, \sum_{i \in I}\lambda^{m^0,D}([0,T])\big)\leq C.$$
\end{lemma}
\begin{proof}
Let $\varepsilon>0$ and $(m^0,D)\in \Omega_\varepsilon$. According to Lemma \ref{bound_A}, there exists a constant $K>0$, independent of $\varepsilon$, such that 
\begin{equation*}
    \sum_{i \in I}\int_0^1   \varphi^{m^0,D}_i(0,s)m^0_i(ds) - \int_0^TD_i(t)\lambda^{m^0,D}_i(dt)<K.
\end{equation*}
Thus, by using same arguments as in the proof of Lemma \ref{borne_lambda} and setting $\tilde{K}:=(K+T\Vert c \Vert_\infty + \Vert g \Vert_\infty)/\varepsilon$, one obtains:
\begin{equation}
\sum_{i\in I}\lambda_i^{m^0,D}([0,T])\leq \tilde{K}.
\end{equation}
By Remark \ref{link_psi_varphi} and the previous inequality, there exists a constant $\tilde{C}$, which depends on $\tilde{K}$, such that $\Vert \varphi^{m^0,D}\Vert_\infty$ and $\Vert \partial_s\varphi^{m^0,D}\Vert_\infty$ are bounded by  $\tilde{C}$. The conclusion follows by setting $C:=\max(\tilde{C},\tilde{K})$.
\end{proof}
We are now ready to prove the Lipschitz regularity of $\mathcal{V}$.
\begin{proof}[Proof of Proposition \ref{lip_prop_ata}]
Let $\varepsilon>0$ and $(m^0,D),(\bar{m}^0,\bar{D})\in \Omega_\varepsilon$. By the definition of $\mathcal{V}$ in \eqref{def_v} and Lemma \ref{bound_phi_lambda_lip_const}, one has:
\begin{equation*}
\begin{array}{ll}
 \mathcal{V}(m^0,D)& \leq \sum_{i\in I}\int_0^1\varphi_i^{m^0,D}\bar{m}^0_i(ds) - \sum_{i\in I}\int_0^T\bar{D}_i(t)\lambda_i^{m^0,D}(dt) + \Vert \partial_s \varphi^{m^0,D} \Vert_\infty \mathcal{W}(m^0,\bar{m}^0)+  \sum_{i\in I}\lambda^{m^0,D}([0,T])\Vert D - \bar{D} \Vert_\infty \vspace{0.2cm}\\
 & 
\leq  \mathcal{V}(\bar{m}^0,\bar{D})
+C \mathcal{D}_\Omega((m^0,D),(\bar{m}^0,\bar{D})),
\end{array}
\end{equation*}
where $C>0$ is a constant defined in Lemma \ref{bound_phi_lambda_lip_const}.
Similarly, we have:
\begin{equation*}
\begin{array}{ll}
 \mathcal{V}(\bar{m}^0,\bar{D})\leq  \mathcal{V}(m^0,D)
+ C \mathcal{D}_\Omega((m^0,D),(\bar{m}^0,\bar{D})).
\end{array}
\end{equation*}
The conclusion follows.
\end{proof}

\appendix
\section{Appendix}\label{appendices_section}

\subsection{Properties of the continuity equation}

Some properties of the weak solution of the continuity equation \eqref{fk} are derived in this subsection. Assumptions in Section \ref{intro_section} are in force in the Appendix. A first result on the support of the solution is established in Lemma \ref{supportm}.

\begin{lemma}
\label{supportm}
For any weak solution $(\alpha,m)$ of \eqref{fk} in the sense of Definition \ref{weak_sol_cont_equ}, $m(t)$ has a support contained in $[0,1]\times I$ for any $t\in[0,T]$.
\end{lemma}
\begin{proof}
Let $\varepsilon>0$ and $\varphi^\varepsilon\in C_c^\infty([0,T]\times \mathbb{R},\mathbb{R}^{\vert I \vert})$ such that, for any $t\in [0,T]$ and $i\in I$,
\begin{equation*}
\begin{array}{c c c c }
\varphi^\varepsilon_i(t,s)\in [0,1], \quad \forall s\in\mathbb{R};\quad
\varphi^\varepsilon_i(t ,s)= 0, \quad \forall s\in\mathbb{R}\setminus (-1-\varepsilon,2+ \varepsilon) ;
&
\mbox{and}
& \varphi^\varepsilon_i(t ,s)= 1 \quad \forall s\in [-1,2].
\end{array}
\end{equation*}
Since $m$ is a weak solution of \eqref{fk}, $b$ satisfies Assumption \ref{hyp_on_b}, and $\varphi_i=\varphi_j$ for any $i,j\in I$, we deduce that, for any $t\in (0,T)$,
\begin{equation}
\label{derivationPhi}
\begin{array}{ll}
\frac{d}{dt}\int_\mathbb{R}\sum_{i\in I}\varphi^\varepsilon_i(t,s)m_i(t,ds)
 &   =  \int_\mathbb{R}\sum_{i\in I}\partial_s\varphi^\varepsilon_i(t,s)b_i(s)m_i(t,ds)\\
 &  =  \int_{-1-\varepsilon}^{-1}\sum_{i\in I} \partial_s\varphi^\varepsilon_i(t,s)b_i(s)m_i(t,ds) +  \int_2^{2+\varepsilon}\sum_{i\in I} \partial_s\varphi^\varepsilon_i(t,s)b_i(s)m_i(t,ds) \\
 & = 0.
\end{array}
\end{equation}
By \eqref{derivationPhi} and the continuity of $m$, we deduce that $t\mapsto \int_\mathbb{R}\sum_{i\in I}\varphi^\varepsilon_i(t,s)m_i(t,ds)$ is constant on $[0,T]$. Let $\varepsilon$ tend to $+\infty$, it holds that $t\mapsto \int_\mathbb{R}\sum_{i\in I}m_i(t,ds)$ is constant over $[0,T]$. Then, we have for any $t\in (0,T)$, 
\begin{equation*}
    \int_\mathbb{R}\sum_{i\in I}m_i(t,ds) = \int_\mathbb{R}\sum_{i\in I}m_i^0(ds)=1.
\end{equation*}
Now let us show that $\int_0^1\sum_{i\in
I}m_i(t,ds)=1$.
Let $\varepsilon>0$ and $\psi^\varepsilon$ be another test function in $C_c^\infty([0,T]\times \mathbb{R},\mathbb{R}^{\vert I \vert})$ such that, for any $t\in [0,T]$ and $i\in I$,
\begin{equation*}
\begin{array}{c c c }
\psi^\varepsilon_i(t ,s)= 0, \quad \forall s\in\mathbb{R}\setminus (-\varepsilon,1+ \varepsilon);
&
 \partial_s\psi^\varepsilon_i(t,s)\geq 0,\quad\forall  s\in (-\varepsilon,0);
 
 & 
 \partial_s\psi^\varepsilon_i(t,s)\leq 0,\quad\forall s\in (1,\varepsilon); \vspace{0.1cm}
 \\
\mbox{and }\,
 \psi^\varepsilon_i(t ,s)= 1, \quad \forall s\in [0,1].&&
\end{array}
\end{equation*}
By the same computation as in \eqref{derivationPhi} and Assumption \ref{hyp_on_b}, one has, for any $t\in (0,T)$,
\begin{equation*}
\begin{array}{ll l}
\frac{d}{dt}\int_\mathbb{R}\sum_{i\in I}\psi^\varepsilon_i(t,s)m_i(t,ds)
 &  =  \int_{-\varepsilon}^{0}\sum_{i\in I} \partial_s\psi^\varepsilon_i(t,s)b_i(s)m_i(t,ds) +  \int_1^{1+\varepsilon}\sum_{i\in I} \partial_s\psi^\varepsilon_i(t,s)b_i(s)m_i(t,ds) \geq 0.
\end{array}
\end{equation*}
Thus, $t\mapsto \int_\mathbb{R}\sum_{i\in I}\psi^\varepsilon_i(t,s)m_i(t,ds)$ is non-decreasing on $[0,T]$. Taking the limit $\varepsilon\to 0$, the map $t\mapsto \int_0^1\sum_{i\in I}m_i(t,ds)$ is also non-decreasing on $[0,T]$. Finally, for any $t\in[0,T]$: $1=\int_0^1\sum_{i\in I}m_i(0,ds)\leq
\int_0^1\sum_{i\in I}m_i(t,ds) \leq \int_\mathbb{R}\sum_{i\in I}m_i(t,ds)=1$.
\end{proof}
For any pair of weak solution $(m,\alpha)$ of \eqref{fk}, the next lemma provides some regularity on $m$ if $\alpha$ and $\partial_s\alpha$ are bounded.
\begin{lemma}\label{regularity_weak_solution}
Let $m^0\in \mathcal{P}([0,1]\times I)$ and $(m,\alpha)$ be a weak solution of \eqref{fk} in the sense of Definition \ref{weak_sol_cont_equ}, with $\alpha\in L^\infty([0,T]\times[0,1],\mathbb{R}^{\vert I \vert\times \vert I \vert }_+)$ . Then, $m$ belongs to $\Lip([0,T],\mathcal{P}([0,1]\times I))$ with a Lipschitz constant independent of $m^0$.
\end{lemma}
\begin{proof}
The proof follows the same steps as the ones of Lemma \ref{UnifCont}. Using that $\alpha$ is uniformly bounded, one can show that, for any $t,\tilde{t}\in [0,T]$,
\begin{equation*}
\mathcal{W}(m(t,\cdot),m(\tilde{t},\cdot))\leq 
\vert t-\tilde{t}\vert(\vert I \vert\Vert b\Vert_\infty + \Vert \alpha \Vert_\infty).
\end{equation*}
The conclusion follows.
\end{proof}

Finally, the next Proposition states that, for any $\alpha$, the existence and uniqueness of an $m$ in $C^0([0,T],\mathcal{P}([0,1]\times I))$, such that $(m,\alpha)$ is a weak solution of \eqref{fk}.
\begin{proposition}\label{weak_sol_fk_coro}
Let $m^0\in \mathcal{P}([0,1]\times I)$ and  $\alpha\in L^\infty([0,T] \times [0,1],\mathbb{R}^{\vert I \vert \times \vert I \vert})$ satisfy $\partial_s \alpha \in L^\infty([0,T] \times [0,1],\mathbb{R}^{\vert I \vert \times \vert I \vert})$. Then, there exists a unique $m\in\Lip([0,T],\mathcal{P}([0,1]\times I)$ such that $(m,\alpha)$ is a weak solution of \eqref{fk} in the sense of Definition \ref{weak_sol_cont_equ}.
\end{proposition}
\begin{proof}
The existence and uniqueness of a weak solution are proved in \cite{cocozza2004approximation} for controls $\alpha$ that are continuous in space and time independent. The extension of this result to bounded controls that are measurable in time is straightforward.
\end{proof}

\subsection{Properties of the metric \texorpdfstring{$\mathcal{D}$}{Lg}}\label{prop_metric_d}
We recall that the function $\mathcal{D}$ is defined on $\mathcal{M}^+([0,T],\mathbb{R}_+^{\vert I \vert})^2$ by
\begin{equation*}
    \mathcal{D}(\lambda,\mu):=    \int_0^T\sum_{i\in I}\left\vert \int_t^T(\lambda_i-\mu_i)(d\tau)\right\vert dt+ \sum_{i\in I}\left\vert  \int_0^T(\lambda_i-\mu_i)(dt)\right\vert .
\end{equation*}
We show in this section that the function $\mathcal{D}$ is a distance on $\mathcal{M}^+([0,T],\mathbb{R}_+^{\vert I \vert})$ (Lemma \ref{D_distance}). Then, a comparison between the weak$^\ast$ topology in $\mathcal{M}^+([0,T],\mathbb{R}_+^{\vert I \vert})$ and the topology induced by the distance $\mathcal{D}$ is presented in Lemma \ref{rem_weak_conv_dist}. Finally, we state in Lemma \ref{density_c_infty} that for any $\delta>0$, the space $C_\delta^0([0,T],\mathbb{R}_+^{\vert I \vert})$ is dense in $\mathcal{M}_{\delta}^+([0,T],\mathbb{R}_+^{\vert I \vert })$ w.r.t. the topology induced by $\mathcal{D}$.

\begin{lemma}\label{D_distance}
    The function $\mathcal{D}$ is a distance on $\mathcal{M}^+([0,T],\mathbb{R}_+^{\vert I \vert})$.
\end{lemma}
\begin{proof}
    One can easily check that for any $\lambda^1,\lambda^2,\lambda^3\in \mathcal{M}^+([0,T],\mathbb{R}_+^{\vert I \vert})$, we have: 
    \begin{equation*}
        \mathcal{D}(\lambda^1,\lambda^2)= \mathcal{D}(\lambda^2,\lambda^1)\geq 0,
    \end{equation*}
    and 
    \begin{equation*}
          \mathcal{D}(\lambda^1,\lambda^2)\leq   \mathcal{D}(\lambda^1,\lambda^3) +   \mathcal{D}(\lambda^3,\lambda^2).
    \end{equation*}
    We need to show that if $  \mathcal{D}(\lambda^1,\lambda^2)=0$, then $\lambda^1=\lambda^2$. By the definition of $\mathcal{D}$ in \eqref{def_metric_D}, if $\mathcal{D}(\lambda^1,\lambda^2)=0$, then, for any $i\in I$, we have $\lambda^1_i([0,T])=\lambda^2_i([0,T])$ and there exists $\mathcal{T}\subset [0,T]$, such that $\mathcal{L}([0,T]\setminus \mathcal{T})=0$ (where $\mathcal{L}$ is the Lebesgue measure on $\mathbb{R}_+$) 
    and for any $t\in \mathcal{T}$, we have $\lambda^1_i([t,T])=\lambda^2_i([t,T])$. By \cite[Proposition 1.8]{ambrosio2000functions}, one gets $\lambda^1_i=\lambda^2_i$ on the $\sigma-$algebra generated by the set $\{[t,T],t\in\mathcal{T}\cup\{0\}\}$, which coincides with $\mathcal{B}([0,T])$. Thus, $\lambda^1$ is equal to $\lambda^2$ on $\mathcal{B}([0,T])$.
\end{proof}

\begin{lemma}\label{rem_weak_conv_dist}
    Let a sequence $\{\lambda^n\}_n$ in $\mathcal{M}_{\delta}^+([0,T],\mathbb{R}_+^{\vert I \vert })$ converge w.r.t. the weak$^\ast$ topology in $\mathcal{M}^+([0,T],\mathbb{R}_+^{\vert I \vert})$ to $\lambda\in\mathcal{M}_{\delta}^+([0,T],\mathbb{R}_+^{\vert I \vert })$. Then we have $\lim_{n\to \infty}\,\mathcal{D}(\lambda^n,\lambda)=0$. 
\end{lemma}
\begin{proof}
Let $\{\lambda^n\}_n$ weakly$^\ast$ converge to $\lambda$ in $\mathcal{M}_{\delta}^+([0,T],\mathbb{R}_+^{\vert I \vert })$. 
Since, for any $f\in C^0([0,1],\mathbb{R}^{\vert I \vert})$,
\begin{equation}
    \label{convergence_weak_f}
    \underset{n\to \infty}{\lim}\sum_{i\in I}\int_0^Tf_i(t)\lambda^n_i(dt)=\sum_{i\in I}\int_0^Tf_i(t)\lambda_i(dt),
\end{equation}
by a slight adaptation of the proof of \cite[Theorem 25.8]{billingsley2013convergence}, one can show that, for any $t\in [0,T]$ such that $\lambda_i(\{t\})=0$ for any $i\in I$, one has $\underset{n\to \infty}{\lim}\,\lambda_i^n([t,T])=\lambda_i([t,T])$. Taking $f=1$ in \eqref{convergence_weak_f}, one obtains $\underset{n\to \infty}{\lim}\,\lambda_i^n([0,T])=\lambda_i([0,T])$. Since $\lambda$ is a finite measure, it has a countable number of atoms on $[0,T]$. 
Thus, for a.e. $t\in [0,T]$ and for $t=0$, one has $\underset{n\to \infty}{\lim}\,\lambda_i^n([t,T])=\lambda_i([t,T])$. Applying the dominated convergence theorem, the conclusion follows.
\end{proof}

\begin{lemma}\label{density_c_infty}
The space $C_\delta^0([0,T],\mathbb{R}_+^{\vert I \vert})$ is dense in $\mathcal{M}_{\delta}^+([0,T],\mathbb{R}_+^{\vert I \vert })$ w.r.t. the topology induced by $\mathcal{D}$.
\end{lemma}
\begin{proof}
    Let $\lambda\in \mathcal{M}_{\delta}^+([0,T],\mathbb{R}_+^{\vert I \vert }) $, and $\tilde{\lambda}\in \mathcal{M}^+([0,T],\mathbb{R}_+^{\vert I \vert})$ be defined from $\lambda$ by $\tilde{\lambda}_i(A):=\lambda_i(A\cap [0,T])$, for any $(i,A)\in I\times \mathcal{B}(\mathbb{R})$. Let $\xi$ be a standard convolution kernel on $\mathbb{R}$ such that $\xi>0$ and $\int_\mathbb{R}\xi(x)dx=1$. Let $\xi^n(t):=\xi(t/\varepsilon_n)/\varepsilon_n$ with $\varepsilon_n\xrightarrow[n\to \infty]{} 0$. For any $n\in \mathbb{N}$, let the function $\lambda^n$ be defined by:
\begin{equation*}
\tilde{\lambda}^n:=\xi^n \ast \tilde{\lambda},
\end{equation*} 
where $\ast$ stands for the convolution product. {\color{black}More preciesely, one has for any $t\in \mathbb{R}$, $n\in \mathbb{N}$ and $i\in I$
\begin{equation}
    \label{def_conv}
    \tilde{\lambda}^n_i(t):=\int_\mathbb{R}\xi^n(t-\tau)\tilde{\lambda}_i(d\tau).
\end{equation}
}
Then, 
$\tilde{\lambda}^n\in C^\infty(\mathbb{R},\mathbb{R}_+^{\vert I \vert})$ and the sequence $\{\tilde{\lambda}^n\}_n$ weakly$^\ast$ converges to $\tilde{\lambda}$ in $\mathcal{M}^+([0,T],\mathbb{R}_+^{\vert I \vert})$ \cite[Lemma 7.1.10]{ambrosio2008gradient}. 
{\color{black}By the definition of the convolution \eqref{def_conv}, one has for any $n\in \mathbb{N}$,
\begin{equation*}
\begin{array}{lll}
   \sum_{i\in I}\int_\mathbb{R}\tilde{\lambda}^n_i(t)dt   &  =  \sum_{i\in I}\int_\mathbb{R}\int_\mathbb{R}\xi^n(t-\tau)\tilde{\lambda}_i(d\tau)dt
     &  =\sum_{i\in I}\int_\mathbb{R}\left(\int_\mathbb{R}\xi^n(t-\tau)dt\right)\tilde{\lambda}_i(d\tau)\vspace{0.2cm} \\
 &    &  =\frac{1}{\varepsilon_n}\sum_{i\in I}\int_\mathbb{R}\left(\int_\mathbb{R}\xi\Big(\frac{t-\tau}{\varepsilon_n}\Big)dt\right)\tilde{\lambda}_i(d\tau)\vspace{0.2cm}\\
 &    &= \sum_{i\in I}\int_\mathbb{R}\tilde{\lambda}_i(d\tau)\vspace{0.2cm}\\
 &    &= \sum_{i\in I}\int_0^T{\lambda}_i(d\tau)\vspace{0.2cm}\\
   &  &\leq \delta.
\end{array}
\end{equation*}
}
Now considering, for any $n\in \mathbb{N}$, the function $\lambda^n$ defined as the restriction of $\tilde{\lambda}^n$ on $[0,T]$, one has that ${\lambda}^n\in C^\infty([0,T],\mathbb{R}_+^{\vert I \vert})$ {\color{black}and by previous inequality}
\begin{equation*}
    \sum_{i\in I}\int_0^T{\lambda}^n_i(t)dt\leq  \sum_{i\in I}\int_0^T \tilde{\lambda}^n_i(t)dt\leq \delta.
\end{equation*}
Thus, ${\lambda}^n\mathcal{L}\in\mathcal{M}_{\delta}^+([0,T],\mathbb{R}_+^{\vert I \vert })$ (where $\mathcal{L}$ is the Lebesgue measure on $[0,T]$) and the sequence $\{{\lambda}^n\mathcal{L}\}_n$ weakly$^\ast$ converges to $\lambda$ in $\mathcal{M}_{\delta}^+([0,T],\mathbb{R}_+^{\vert I \vert })$. By  Remark \ref{rem_weak_conv_dist} we have, $\lim_{n\to \infty}\mathcal{D}(\lambda,\lambda^n\mathcal{L})=0$.\end{proof}

\subsection{Proof of Lemma 4.8}\label{proof_lemma_eq_weak_sol_int_sol}
This section is devoted to the proof Lemma \ref{int_sol_imp_weak_sol} stated in Section \ref{existence_measure_weak_solution}. 

\begin{proof}[Proof of Lemma \ref{int_sol_imp_weak_sol}]

    Let $\varphi$ be a weak solution of \eqref{hjb_0} associated with $\lambda$. Let $\beta\in C^1([0,T]\times[0,1], \mathbb{R}^{\vert I \vert})$ be a test function. Then,
\begin{equation}
\label{weak_sol_smooth_case_beta}
\begin{array}{l}
\int_0^1\varphi_i(0,s)\beta_i(0,s)ds -\int_0^1g_i(s)\beta_i(T,s)ds
+\int_0^T\int_0^1\bigg(\partial_t\beta_i(t,s)+\partial_s(\beta_i(t,s)b_i(s))\bigg)\varphi_i(t,s)dsdt \vspace{0.1cm} \\
+\int_0^T\int_0^1
\left(\sum_{j\in I,j\neq i}H(\varphi_i(t,s)-\varphi_j(t,s))-c_i(t,s)\right)\beta_i(t,s)dtds-\int_0^T\int_0^1\beta_i(t,s)ds\lambda_i(dt)\vspace{0.1cm}\\
= 0.
\end{array}
\end{equation}
We choose the function $\beta$ such that there exist $\theta\in C^\infty([0,T]\times [0,1],\mathbb{R}^{\vert I \vert})$ and $\xi\in C^\infty([0,1],\mathbb{R}^{\vert I \vert})$ satisfying: 
\begin{equation}
\label{equ_verif_phi}
\begin{array}{ll}
\partial_t\beta_i(t,s)+\partial_s(\beta_i(t,s)b_i(s)) = \theta_i(t,s) &  \mbox{ for any }(t,s,i)\in (0,T)\times (0,1)\times I,\\
 \beta_i(0,\cdot)=\xi_i(\cdot)&\mbox{ for any }(s,i)\in [0,1]\times I.
\end{array}
\end{equation}
The function $\beta$ is given by for any $(t,s,i)\in [0,T]\times [0,1]\times I$,
\begin{equation*}
\beta_i(t,s)=\int_0^t\theta_i(\tau, S_i^{t,s}(\tau))\exp\left(-\int_\tau^t b'_i(S_i^{t,s}(r))dr\right)d\tau
+ \xi_i(S_i^{t,s}(0))
\exp\left(-\int_0^t b'_i(S_i^{t,s}(\tau))d\tau\right),
\end{equation*}
where $S_i^{t,s}$ is the unique solution of the ODE \eqref{ODE}.
To simplify \eqref{weak_sol_smooth_case_beta}, we introduce, for any $i\in I$,
 the functions $\nu_i$ and $\pi_i$, satisfying $\beta_i=\nu_i+\pi_i$, and for any $(t,s)\in [0,T]\times [0,1]$,
\begin{equation*}
\nu_i(t,s):=\int_0^t\theta_i(\tau, S_i^{t,s}(\tau))\exp\bigg(-\int_\tau^t b'_i(S_i^{t,s}(r))dr\bigg)d\tau
\quad \mbox{ and }\quad
\pi_i(t,s):=\xi_i(S_i^{t,s}(0))
\exp\bigg(-\int_0^t b'_i(S_i^{t,s}(\tau))d\tau\bigg).
\end{equation*}
Setting $h_i(t,s) := \sum_{j\in I,j\neq i }H(\varphi_i(t,s)-\varphi_j(t,s)) - c_i(t,s)$ for any $(t,s,i)\in [0,T]\times [0,1]\times I$, we have $h\in L^1((0,T)\times (0,1))$. 
By the definition of $\beta,\nu,\pi$ and $h$, equality \eqref{weak_sol_smooth_case_beta} becomes, for any $i\in I$,
\begin{equation}\label{weak_sol_smooth_case_beta_2}
\begin{array}{l}
  \int_0^1\varphi_i(0,s)\xi_i(s)-g_i(s)(\nu_i(T,s)+\pi_i(T,s))ds
         +\int_0^T\int_0^1\theta_i(t,s)\varphi_i(t,s)+h_i(t,s)(\nu_i(t,s)+\pi_i(t,s))dsdt \vspace{0.2cm}  \\
         -\int_0^T\int_0^1\beta_i(t,s)ds\lambda_i(dt)=0.
\end{array}
\end{equation}
We want to compute each integral of the previous equality.
By the definition of $\nu$ and by switching the order of  integration, one has
\begin{equation*}
\begin{array}{ll}
         \int_0^T\int_0^1h_i(t,s)\nu_i(t,s)dtds &=\int_0^T\int_0^1\int_0^t h_i(t,s)\theta_i(\tau, S_i^{t,s}(\tau))\exp\bigg(-\int_\tau^t b'_i(S_i^{t,s}(r))dr\bigg)d\tau ds\,dt \\
     & =\int_0^T\int_0^1\int_\tau^T h_i(t,s)\theta_i(\tau, S_i^{t,s}(\tau))\exp\bigg(-\int_\tau^t b'_i(S_i^{t,s}(r))dr\bigg)dt ds\,d\tau. 
\end{array}
\end{equation*}
The function $s\mapsto S_i^{t,s}(\tau) $ being in $C^1([0,1])$, we consider the change of variable: $x= S_i^{t,s}(\tau)$, $\frac{ds}{dx}=\partial_x S_i^{t,s}(\tau)$. By the definition of  $S_i^{t,s}$ in \eqref{ODE}, one has $s= S_i^{\tau,x}(t)$ and $S_i^{t,S_i^{\tau,x}(t)}(r)=S_i^{\tau,x}(r)$. 
Thus, previous equality becomes
\begin{equation*}
 \int_0^T\int_0^1h_i(t,s)\nu_i(t,s)dtds 
 = \int_0^T\int_0^1\theta_i(\tau, x)\int_\tau^T h_i(t,S_i^{\tau,x}(t))\exp\bigg(-\int_\tau^t b'_i(S_i^{\tau,x}(r))dr\bigg)\partial_xS_i^{\tau,x}(t)dt\, dx\,d\tau.
\end{equation*}
Since $b\in C^1([0,1],\mathbb{R}^{\vert I \vert})$, the flow $S_i$, satisfies the following equation for any $(i,\tau,t,s)\in I\times (0,T)\times (0,T)\times (0,1)$:
\begin{equation*}
\partial_xS_i^{\tau,x}(t)=\exp\bigg(\int_\tau^tb'_i(S_i^{\tau,x}(r))dr\bigg).
\end{equation*}
One deduces
\begin{equation}
\label{eq2_demo_weak_int}
 \int_0^T\int_0^1h_i(t,s)\nu_i(t,s)dtds =
\int_0^T\int_0^1\theta_i(\tau, x)
\int_\tau^T h_i(t,S_i^{\tau,x}(t))dt\, dx\, d\tau.
\end{equation}

By the same computations, for any $i\in I$, one has:
\begin{equation}
\begin{array}{ll}\label{eq_demo_weak_int_b}
 \int_0^1 g_i(s)\nu_i(T,s)ds
 & 
  = \int_0^T\int_0^1
 \theta_i(\tau, x) g_i(S_i^{\tau,x}(T))ds d\tau,
\end{array}
\end{equation}
\begin{equation}
\label{eq3_demo_weak_int}
\begin{array}{ll}
 \int_0^T\int_0^1h_i(t,s)\pi_i(t,s)dtds
&= \int_0^1\xi_i(x)\int_0^Th_i(t,S_i^{0,x}(t)) dt dx,\\
\end{array}
\end{equation}
\begin{equation}
\label{eq4_demo_weak_int}
\begin{array}{ll}
 \int_0^1 g_i(s)\pi_i(T,s)ds
& = \int_0^1 \xi_i( x) g_i(S_i^{0,x}(T))dx,
\end{array}
\end{equation}
and:
\begin{equation}
\label{same_calc_lambda}
\begin{array}{l}
 \int_0^T\int_0^1\beta_i(t,s)ds\lambda_i(dt)
 =  \int_0^T 
 \int_0^1\theta_i(\tau,x)L_i^\lambda(\tau)dxd\tau
+ \int_0^1\xi_i(x)L_i^\lambda(0)dx.
\end{array}
\end{equation}
By the definition of $h$ given above and \eqref{eq2_demo_weak_int}-\eqref{same_calc_lambda}, \eqref{weak_sol_smooth_case_beta_2} becomes
\begin{equation*}
\begin{array}{l}
 \int_0^1 \xi_i(s)\bigg(\varphi_i(0,s)+\int_0^T\sum_{j\in I,j\neq i }H((\varphi_i-\varphi_j)(\tau, S_i^{0,s}(\tau))) -c_i(\tau,S_i^{0,s}(\tau)) d\tau -L_i^\lambda(0)-g_i(S_i^{0,s}(T))\bigg)ds  \\
+
 \int_0^T \int_0^1 \theta_i(t,s)\bigg(\varphi_i(t,s) 
+
\int_t^T \sum_{j\in I,j\neq i }H((\varphi_i-\varphi_j)(\tau, S_i^{t,s}(\tau))) -c_i(\tau,S_i^{t,s}(\tau))d\tau - L_i^\lambda(t) -g_i(S_i^{t,s}(T)) \bigg)dsdt\\
=0.
\end{array}
\end{equation*}
This equality holds for any test functions $\theta_i$ and $\xi_i$. Then, one has for any $s\in [0,1]$:
\begin{equation*}
\varphi_i(0,s) = \int_0^T\sum_{j\in I,j\neq i }-H((\varphi_i-\varphi_j)(\tau, S_i^{0,s}(\tau)))+c_i(\tau,S_i^{0,s}(\tau)) d\tau+L_i^\lambda(0) +g_i(S_i^{0,s}(T)),
\end{equation*}
and for any $(t,s)\in [0,T]\times [0,1]$:
\begin{equation*}
\varphi_i(t,s)  = \int_t^T \sum_{j\in I,j\neq i }-H((\varphi_i-\varphi_j)(\tau, S_i^{t,s}(\tau)))+c_i(\tau,S_i^{t,s}(\tau))d\tau +L_i^\lambda(t)+g_i(S_i^{t,s}(T)).
\end{equation*}
\end{proof}

\section*{Acknowledgement}

This research benefited from the support of the FiME Lab (Institut Europlace de Finance) and the support of the FMJH Program Gaspard Monge for optimization and operations research and their interactions with data science.
I would like to thank my advisor, Pierre Cardaliaguet, for suggesting the problem and many helpful conversations.
\bibliographystyle{plain}
\bibliography{biblio}

\end{document}